\documentclass{amsart}
\usepackage[utf8]{inputenc}

\usepackage{xcolor}
\usepackage{amsmath}
\usepackage{amssymb}
\usepackage{comment}

\usepackage{geometry}
\usepackage{hyperref}

\allowdisplaybreaks

\title{Neural Networks in non-metric spaces}
\author{Luca Galimberti}
\address{Luca Galimberti \\ 
King's College London\\
Department of Mathematics\\
Strand Building, WC2R 2LS, London, UK}
\email[]{luca.galimberti\@@kcl.ac.uk}

\newtheorem{theorem}{Theorem}[section]
\newtheorem{definition}[theorem]{Definition}
\newtheorem{lemma}[theorem]{Lemma}
\newtheorem{proposition}[theorem]{Proposition}

\newtheorem{corollary}[theorem]{Corollary}

\newtheorem{remark}[theorem]{Remark}
\newtheorem{example}[theorem]{Example}
\newtheorem{Assumption}{Assumption}

\newcommand{\R}{\mathbb R}

\newcommand{\N}{\mathbb N}

\newcommand{\norm}[1]{\left\lVert#1\right\rVert}
\newcommand{\abs}[1]{\left |#1\right|}
\newcommand{\X}{\mathfrak X}
\newcommand{\Y}{\mathfrak Y}

\DeclareMathOperator{\Span}{span}

\date{\today}

\begin{document}
\begin{abstract}
Leveraging the infinite dimensional neural network architecture we proposed in \cite{BDG0} and which can process inputs from Fr\'echet spaces, and using the universal approximation property shown therein, we now largely extend the scope of this architecture by proving several universal approximation theorems for a vast class of input and output spaces. More precisely, the input space $\X$ is allowed to be a general topological space satisfying only a mild condition (``quasi-Polish''), and the output space can be either another quasi-Polish space $\Y$ or a topological vector space $E$. Similarly to \cite{BDG0}, we show furthermore that our neural network architectures can be projected down to ``finite dimensional'' subspaces with any desirable accuracy, thus obtaining approximating networks that are easy to implement and allow for fast computation and fitting. The resulting neural network architecture is therefore applicable for prediction tasks based on functional data. To the best of our knowledge, this is the first result which deals with such a wide class of input/output spaces and simultaneously guarantees the numerical feasibility of the ensuing architectures. Finally, we prove an obstruction result which indicates that the category of quasi-Polish spaces is in a certain sense the correct category to work with if one aims at constructing approximating architectures on infinite-dimensional spaces $\X$ which, at the same time, have sufficient expressive power to approximate continuous functions on $\X$, are specified by a finite number of parameters only and are ``stable'' with respect to these parameters.

\end{abstract}

\maketitle

\tableofcontents

\section{Introduction}

The study of neural networks on finite-dimensional Euclidean spaces can be traced back to the seminal paper \cite{McCullochPitts} by McCulloch and Pitts. The overall idea of this work was to imitate the functioning of the human brain with a system consisting of various
connections and neurons, where data is fed in, processed and finally returned as
output. In mathematical terms, such a object can be more conveniently described by a concatenation of affine and non-linear maps, where the affine maps represent the connections between the different
neurons while the non-linear maps the transformation of the input data. The well-known universal approximation theorem, first stated and shown by \cite{Cybenko1989} and \cite{Hornik1991}, ensures that such neural networks can approximate arbitrary well, uniformly on compact sets, any continuous function from $\R^d$ to $\R$. More precisely, for a fixed continuous function $\sigma : \mathbb{R}\rightarrow \mathbb{R}$ (the {\em activation function}) and $a \in \mathbb{R}^d, L , b \in \mathbb{R}$, a {\em neuron} is a function $\mathcal{N}_{L,a,b}\in C(\mathbb{R}^d)$
defined by 
\begin{equation}\label{eq: finite-dim neurons}
  x\mapsto L \sigma (a^{\top}x +b),   
\end{equation}
and a {\em one layer neural network} is a finite sum of neurons
\begin{equation}\label{eq: finite-dim nn}
\R^d\ni x\mapsto
\mathcal{N}(x) =
\sum_{j=1}^{J} \mathcal{N}_{L_j,a_j,b_j}(x).
\end{equation}
The universal approximation theorem then establishes conditions on $\sigma$ such that the set of one layer neural networks, which corresponds to the linear space of functions generated by the neurons
\begin{equation}\label{eq: space generated by finite-dim neurons}
\left\{
\sum_{j=1}^{J} \mathcal{N}_{L_j,a_j,b_j};\, J\in\N,L_j\in\R,a_j\in\R^d,b_j\in\R\right\}
\end{equation}
is dense in $C(\R^d)$ with respect to the topology of uniform convergence on compacts. 

To this end, the most widely known property of $\sigma$ that was shown in \cite{Cybenko1989} and \cite{HORNIK1989359} to lead to the density in $C(\mathbb{R}^d)$ of the linear span of neurons above is the {\em sigmoid} property: this requires that $\sigma$ admits limits at $\pm\infty$, i.e. $\lim_{t\rightarrow \infty} \sigma(t)=1$ and $\lim_{t\rightarrow -\infty} \sigma(t)=0$. This condition was later relaxed to a boundedness condition \cite{FUNAHASHI1989183} and a non-polynomial condition \cite{LESHNO1993861}.We point out that all these results pertain to finite-dimensional shallow neural networks consisting of one or two layers with many neurons (bounded depth, arbitrary width). In contrast stands the analysis of networks with arbitrary depth and bounded width which has also attracted a lot of attention recently \cite{articleWidth1,articleWidth2,pmlr-v125-kidger20a}. For a general overview of the earlier literature on the approximation theory of neural network we refer the reader to \cite{pinkus_1999} and for a more recent account to \cite{MathFoundationsDeep}. 

Lately, there has been a surging interest in neural networks with infinite-dimensional input and output spaces, e.g. suitable classes of infinite-dimensional vector spaces and metric spaces. For example, early instances of that can be found in \cite{392253}, where, among other results, continuous functions $f:U\to \R$ are approximated with suitable architectures: here $U$ denotes a compact subset of the Banach space of continuous functions $C(K)$, where $K$ is some compact subset of $\R^n$. Another early instance of this is provided by \cite{STINCHCOMBE1999467}, where continuous functions on locally convex spaces are analyzed. More recently, in \cite{BDG0,GKL} we considered functions and neural architectures which can process inputs from Fr\'echet spaces, while in \cite{CuchieroSchmockerTeichmann} the authors studied approximation capabilities of neural networks defined on infinite-dimensional weighted spaces. We will come back and review in more detail these results and further more in Subsection \ref{subsec: literature} below.

In the present paper, by leveraging the infinite dimensional neural networks we introduced in \cite{BDG0}, and using its universal approximation capability shown therein, we significantly extend the scope of this architecture by proving several universal approximation theorems for a vast class of infinite-dimensional input and output spaces. More precisely, the input space $\X$ is allowed to be a general topological space satisfying only a mild condition (``quasi-Polish'') which is very often met in practice, and the output space can be either another quasi-Polish space $\Y$ or a topological vector space $E$. By quasi-Polish we mean that the underlying space $\X$ admits the existence of a countable family of real valued continuous functions which separate points. The architecture we are going to construct will be constituted by two main parts: the first element is a linear superposition of infinite-dimensional neural networks from \cite{BDG0} defined on the separable Hilbert space $V:=\ell^2(\N)$, and fully specified by a set of trainable parameters; the second element consists of a continuous and injective map $F$ from $\X$ into $V$, whose existence is guaranteed by the ``quasi-Polish'' assumption. This map $F$ encodes the underlying topological and geometrical structure of the space $\X$; most of the times, it can be specified by the user, and it is not part of the trainable parameters. In other words, we transport via the map $F$ the input $x$ from the underlying space $\X$ (which is only a topological space and hence lacking any linear structure in general) to the more favorable Hilbert space $V$. Our ensuing architectures are then obtained by pre-composing the neural nets defined on $V$ with the continuous injection $F$. Finally, the class of activation functions $\sigma:V\to V$ which ensures the validity of these results is very wide and flexible, and can be seen as the proper generalization to an infinite-dimensional setting of the sigmoid property for functions from $\R$ to $\R$ recalled above: we refer to Subsection \ref{subsec: Infinite-dimensional neural networks} for further details.

While such approximation results might be of independent interest, from a practical point of view it is not clear a priori how these resulting architectures that involve infinite dimensional inputs and outputs as well as an infinite number of trainable parameters can actually be programmed in a machine with finite computational power and memory. We therefore address the important question of approximating our architectures by finite dimensional, easy to calculate quantities. In the scalar case, namely when the output space of the neural networks is $\R$, upon imposing the additional mild condition that the activation function $\sigma$ is Lipschitz, we are able to obtain neural networks which possess an architecture similar in spirit to the classical feedforward ones, with the only exception that the activation function is now allowed to be multidimensional. In the vector case, namely when the target is an arbitrary topological real vector space $E$, by the very nature of the problem, a second layer of approximation was necessary in the construction of our architectures and to accomplish a universal approximation result. Nonetheless, despite this additional element that makes the architectures more involved, also in the present case, under the assumption that the topological vector space $E$ admits a pre-Schauder basis, we could eventually obtain once more neural nets similar in spirits to the classical multi-layer perceptrons. Therefore, in all cases, we end up with architectures specified by a finite number of parameters which permit for an easy to calculate gradient, a crucial step for training the networks via a back-propagation algorithm. This is in stark contrast with other infinite-dimensional approximation results found in the literature, as in e.g. \cite{STINCHCOMBE1999467}, where in general the resulting architectures therein in many cases are specified by an infinite number of trainable parameters: refer to Subsection \ref{subsec: literature} for a deeper analysis. 

Possible applications of our results are within the area of machine learning, in particular in the many situations where the input of each sample in the training set is actually a function. This comprises \textit{i)} functional data analysis (see e.g. \cite{RamSilver} for an account on this subject and examples); \textit{ii)} learning the solution of a partial differential equation, where we refer to the works on physics-informed neural networks in \cite{PhysicsInformed}, DeepONets in \cite{10.1093/imatrm/tnac001,LiDeepONets}, neural operators in \cite{https://doi.org/10.48550/arxiv.2108.08481}, and the papers by e.g. \cite{Han8505,dneuralnetPDE,Cuchiero2020DeepNN,beck2021overview}; \textit{iii)} mathematical finance, e.g. \cite{BDG2}; \textit{iv)} approximation of infinite-dimensional dynamical systems, as in the echo-state networks and reservoir computing
  \cite{EchoState,ReservoirComputing} and in the so-called metric hypertransformers \cite{Hypertransformer}. Moreover, there are other instances where functional data appears naturally. For example grey scale images can be understood as a function $I: [0,1]^2 \rightarrow [0,1]$, and therefore for imagine classification or recognition problems (see \cite{https://doi.org/10.48550/arxiv.2201.11440,brainImage}) one is now interested in approximating the function $f$ that assigns to each image its classification $f(I)$.

Finally, we prove an obstruction result which indicates that the category of quasi-Polish spaces is in some sense the correct category to work with if one aims at constructing approximating architectures on infinite-dimensional spaces $\X$ (topological, algebraic,...) which, at the same time, have sufficient expressive power to approximate continuous functions on $\X$; which are implementable in practice because specified by a finite number of parameters only; and that are ``stable'' with respect to these parameters. These requirements are natural, because they demand respectively that \textit{i)} the approximating architectures satisfy a universal approximation property; \textit{ii)} they can be represented and implementable into a machine with finite memory and computing power; \textit{iii)} they enjoy a stability under small perturbations of their ``tuned parameters'', something which is very useful to have in practice. Broadly speaking (see Proposition \ref{lemma: uat + finite} for a more precise statement), we will prove that if a topological space $(\X,\tau)$ grants the existence of such architectures, then it must be necessarily quasi-Polish.

\subsection{Related literature and comparison with other results}\label{subsec: literature} 

The approximation with neural networks of functions defined on some possibly infinite dimensional space $\X$ probably goes back to \cite{76498}, where in the context of discrete time systems, non-linear functionals on a space of functions from $\mathbb{N} \cup \{ 0\}$ to $\mathbb{N}$ are approximated with neural networks. In \cite{6795742} the authors derive networks that approximate the functionals on the function spaces $L^p ([-1,1]^d)$ for $1 \leq p < \infty $ and $C ([-1,1]^d)$ for $d\in \N$. As already recalled above, in \cite{286886, 392253} the authors consider the approximation of non-linear operators defined on infinite dimensional spaces and use these results for approximating the output of dynamical systems. In \cite{STINCHCOMBE1999467}, the author could approximate, uniformly on compacts, real valued continuous functions defined on locally convex spaces via suitable architectures. In \cite{256500}, neural networks defined on Hilbert spaces were considered, and density results in the $L^2$-sense for these architectures were shown. \cite{Guss} proves the universal approximation property for two-layer infinite dimensional neural networks, and they show their approximation property for continuous maps between spaces of continuous functions on compacts.

Very recently, in \cite{Hypertransformer} the authors have considered hyper-transformers architectures with the aid of which they have been able to approximate maps $f:K\subset\R^d\to (Y,\rho)$, whereas $K$ is a compact subset, $(Y,\rho)$ is a suitable metric space and $f$ is assumed to be $\alpha$-H\"older, $0<\alpha\le 1$. The result is very interesting, since to our knowledge this is one of the first that considers as output space a quite general metric space. However, this result still falls into the category of intrinsically ``finite-dimensional'' results, as it can be seen here: since the authors consider only maps of $\alpha$-H\"older regularity (and not merely continuous functions), it follows that the Hausdorff dimension of $f(K)$ can be at most $d/\alpha$, and thus $f(K)$ can be homeomorphically embedded into the Euclidean space $\R^{1+2d/\alpha}$ by the celebrated Menger-N\"obeling theorem (refer to e.g. \cite[Theorem 1.12]{Robinson}). More challenging would be to deal with an arbitrary compact metric space $(K,d)$ and a continuous map with no extra regularity $f:(K,d)\to (Y,\rho)$ where $Y$ is another metric space, because in that case the above argument would not be valid anymore in general (just think of a space-filling curve and the related Hahn-Mazurkiewicz theorem). In this paper, we address this answer: refer to Remark \ref{rmk: Acciaio's problem}.

Always very recently, \cite{CuchieroSchmockerTeichmann} have studied approximation capabilities of neural networks defined on infinite-dimensional weighted spaces, obtaining global universal approximation results for continuous functions whose growth is controlled by a weight function. Crucial step in their proof (as well as in previous classical results) is some version of Stone-Weierstrass theorem. This result is then used to get a universal approximation theorem for architectures consisting of an additive family $\mathcal{H}$ as hidden layer maps and a non-linear activation function applied to each hidden layer. We recall here that an additive family is a collection of real valued separating points continuous functions that is closed under addition and contains the constants functions. An essential element here which must be stressed is that in this kind of architectures the elements of the additive family $\mathcal{H}$ are parts of the ``trainable parameters'', namely the optimization algorithm is required to select, among other things,  $f_1,\dots,f_M$ elements of $\mathcal{H}$ during the training procedure (exactly as in the case for plain vanilla neural networks where biases and weights are chosen by the stochastic gradient descent procedure). In our case, however, as aforementioned, the sequence of functions $(f_n)_{n\in\N}$ which separates points is not part of the trainable parameters, but it is specified by the user. The only trainable parameter associated to the sequence is an integer $N_\ast$ which selects the first $N_\ast$ elements $f_1,\dots, f_{N_\ast}$ of the sequence: refer to e.g. Theorem \ref{thm: second}. In contrast to most currently available neural networks for infinite spaces, our architecture focuses on information inherited by the decomposition of the input $x\in\X$ via the map $F$ (namely the separating sequence $(f_n)_{n\in\N}$). This decomposition carries important structural information that helps in the learning process and, in the very particular case where the separating sequence is induced by a Schauder basis (see Example \ref{ex: Frechet spaces carrying a Schauder basis}), this set of ideas was used in our previous results \cite{BDG0,GKL} (in the context of PDEs, a slightly similar approach has appeared with Fourier Neural Operators \cite{FNO}).

Other recent results are the so-called DeepONets for the approximation of operators between Banach spaces of continuous functions on compact subsets of $\mathbb{R}^n$: they have been proposed and analyzed in \cite{LiDeepONets, 10.1093/imatrm/tnac001}. DeepONets follow a similar structure as the one used in \cite{392253} of a branch net that uses signals to extract information about the functions in the domain, and a trunk net to map to the image. In DeepONets both the branch and trunk nets are deep neural nets. In  \cite{https://doi.org/10.48550/arxiv.2108.08481} and \cite{https://doi.org/10.48550/arxiv.2003.03485}, the authors propose a neural network method to approximate the solution operator that assigns to a coefficient function for a partial differential equation (PDE) its solution function. This leads again to the approximation of an operator between Banach spaces of functions that are defined on a bounded domain in $\mathbb{R}^n$. The neural network that is presented in \cite{https://doi.org/10.48550/arxiv.2108.08481,https://doi.org/10.48550/arxiv.2003.03485} is tailor-made for the specific problem at hand and its structure is motivated by the Green function which defines the solution to the PDE.

Finally, infinitely wide neural networks, with an infinite but countable number of nodes in the hidden layer have been studied in the context of Bayesian learning, Gaussian processes and kernel methods by several authors, see e.g., \cite{Neal,Williams,ChoSaul,HJ}.

\subsection{Outline} The outline for the paper is as follows. In Section \ref{sec: Preliminaries}, after introducing all the relevant notations, we give a primer on quasi-Polish spaces and numerous examples thereof; finally we recall the relevant material from \cite{BDG0}  which will be needed afterwards. In Section \ref{sec: Quasi-Polish neural networks architectures}, we rigorously introduce our infinite-dimensional neural network architectures on which the subsequent universal approximation theorems will be based. In Section \ref{sec: Main results}, we will prove our approximation results, and in Section \ref{sec: On the necessity of the quasi-Polish condition} we will present and prove an obstruction theorem.


\section{Preliminaries}\label{sec: Preliminaries}
In an attempt to make this paper more self-contained, we fix the relevant notation and briefly review some basic aspects of functional analysis and topology; give a thorough introduction to quasi-Polish spaces and provide several instances thereof; describe the neural network architectures from \cite{BDG0} which we are going to build on our results.

\subsection{Notation and conventions}

\begin{itemize}
    \item $\N=\{1,2,3,\dots\}$.
    \item All vector spaces are assumed to be real. For a given vector space $E$, $E^\ast$ will denote its algebraic dual.
    \item A topological vector space $(E,\tau_E)$ is a topological space that is also a vector space and for which the vector space operations of addition and scalar multiplication are continuous. It need not be Hausdorff.
    \item If $E$ and $F$ are two topological vector spaces, then $\mathcal{L}(E,F)$ will denote the space of linear and continuous maps from $E$ to $F$, and $\mathcal{L}(E):=\mathcal{L}(E,E)$. 
    \item If $E$ is a topological vector space, then $E':=\mathcal{L}(E,\R)$ is its topological dual.
    \item If $E$ is a topological vector space, then $\sigma(E,E')$ and $\sigma(E',E)$ will denote respectively the weak topology on $E$ and the weak-star topology on $E'$.
    \item A topological vector space $E$ is a locally convex space if its topology is determined by a family of seminorms $\{p_\lambda;\,\lambda\in\Lambda\}$ on it. This topology is Hausdorff if and only if $\cap_{\lambda\in\Lambda}\{x\in E;\, p_\lambda(x)=0\}=\{0\}$.
    \item For a topological vector space $E$, $\langle \cdot, \cdot \rangle_{E^\ast,E}$ will denote the duality between $E^\ast$ and $E$, i.e. the pairing between $E^\ast$ and $E$. When there is no possibility of confusion, we will omit $E^\ast$ and $E$ from the symbol and simply write $\langle\cdot,\cdot\rangle$.
    \item A locally convex space $E$ is called Fr\'echet if it is metrizable and complete. This happens precisely when the familiy of seminorms inducing its topology is countable. This family, say $\{p_n\}_{n\in\N}$, can be assumed increasing, and a compatible metric is then given by 
    \begin{equation}\label{eq: metric of Frechet}
    \rho(x,y):= \sum_{n=1}^\infty 2^{-n}\frac{p_n(x-y)}{1+p_n(x-y)},\quad x,y\in E.
    \end{equation}
    \item $V:=\ell^2(\N)=$ the Hilbert space of square-integrable real sequences $a=(a_j)_{j\in\N}$ with its natural scalar product $(a,b)_V=\sum_{j=1}^\infty a_jb_j$.
    \item A topological vector space $E$ is said to have a pre-Schauder basis $(s_k)_{k\in\N}\subset E$ provided that for every $x\in E$ there exists a unique sequence $(x_k)_{k\in\N}\subset\R$ such that $\sum_{k=1}^\infty x_ks_k$ converges to $x$. 
     We can then define the canonical linear projectors $\beta^E_k$ associated to the basis, namely
    \[
    \beta^E_k:E\to \R,\quad x\mapsto x_k,\quad k\in\N.
    \]
    We can also define the linear operators $\Pi_N$
    \[
    \Pi^E_N:E\to\Span\{s_1,\dots,s_N\},\quad x\mapsto\sum_{k=1}^Nx_ks_k=\sum_{k=1}^N\langle\beta^E_k,x\rangle s_k,\quad N\in\N,
    \]  
    which are called projections operators. If $\beta^E_k,k\in\N,$ are continuous, then also the projections $\Pi_N$ are continuous, and in this situation we say that $(s_k)_{k\in\N}$ is a Schauder basis: see \cite{PreSchauder}.

     If $E$ is now assumed to be Fr\'echet (with seminorms $\{p_n\}_{n\in\N}$)  and $(s_k)_{k\in\N}\subset E$ is a Schauder basis, then we have that for any $\mathcal{K}\subset E$ compact and $n\in\N$ one has
    \[
    \sup_{x\in\mathcal{K}} p_n(x-\Pi^E_Nx)\to 0,\quad\text{as } N\to\infty.
    \]   
    In the particular case when $E=V$, we will always write $\Pi_N$ rather than $\Pi_N^V$.
    \item Given a topological space $(\X,\tau)$, $\mathcal{B}(\X)$ will denote as usual its Borel sigma-algebra.
    \item Given two topological spaces $(\X,\tau_\X)$ and $(\Y,\tau_\Y)$, $C(\X,\Y)$ will denote the space of continuous functions from $\X$ to $\Y$. If $\Y=\R$ with the Euclidean topology, we will write $C(\X):=C(\X,\R)$.
    \item Given a topological space $(\X,\tau)$, a non-empty subset $T\subset\X$, and a topological vector space $E$, we will denote by $C(T)\otimes E$ the set of all finite sums
    \[
    \sum_i a_i \otimes z_i
    \]
    with $a_i\in C(T)$ and $z_i\in E$, and where $a_i\otimes z_i$ means the function $t\mapsto  a_i(t)z_i$. Clearly, $C(T)\otimes E$ is a vector subspace of $C(T,E)$.
    \item Given a measure space $(\Omega,\mathcal{A},\mu)$ and $p\in[1,\infty]$, as usual $L^p(\mu)=L^p(\Omega,\mathcal{A},\mu)$ will denote the $p$-$th$ Lebesgue space.
    \item A subset $C$ of a vector space $L$ is circled if $\lambda C\subset C$ whenever $\abs{\lambda}\le 1$. A subset $U$ of $L$ is radial if for every finite set of points $F\subset L$ there exists $\lambda_0\in\R$ such that $F\subset\lambda U$ whenever $\abs{\lambda}\ge \abs{\lambda_0}$.

\end{itemize}

\subsection{A primer on quasi-Polish spaces}
We are now going to introduce the class of input spaces on which we will construct our architectures and for which we will be able to prove our universal approximation results. 

Let $(\X,\tau)$ be an arbitrary topological space. The minimal assumption we are going to use in the whole paper is the following one borrowed from Jakubowski \cite{jakubowski}:

\begin{Assumption}\label{assumption}
There exists a countable family $\{h_i:\X\to [-1,1]\}_{i=1}^\infty$ of $\tau$-continuous functions which separate points of $\X$, namely for each $x_1,x_2\in\X$ with $x_1\neq x_2$ there exists $i\in \N$ such that
\[
h_i(x_1) \neq h_i(x_2).
\]
Such a family will be called a separating sequence.
\end{Assumption}
\begin{definition}
A topological space $(\X,\tau)$ which satisfies Assumption \ref{assumption} for some family $(h_i)_{i=1}^\infty$ will be called \textbf{quasi-Polish}. When we want to stress the sequence $(h_i)_{i=1}^\infty$ in the definition of $\X$, we will write $(\X,\tau,(h_i)_{i=1}^\infty)$.
\end{definition}
The assumption is very simple, and it is satisfied by a huge class of topological spaces. For instance all Polish spaces fall into this category, i.e. Polish spaces are quasi-Polish, as we will see below; besides, quasi-Polish spaces inherit many properties of Polish spaces. This kind of spaces were originally introduced by Jakubowski \cite{jakubowski}, and since then they have gained increased popularity, especially in SPDEs theory (see e.g. \cite{BrzezniakMotyl,BrzezniakOndrejat,SmithTrivisa}). One of their remarkable properties (proved by Jakubowski) is the validity of the Skorokhod's representation theorem (refer to e.g. \cite{daprato_zabczyk_1992}), which makes them a very versatile tools in problem pertaining to SPDEs and infinite dimensional stochastic analysis when the ambient is a non-metric space. For a nice and compact collection of the main properties of these spaces we refer to the Appendix of the recent paper \cite{GalimbertiHoldenKarlsenPang}.

Assumption 1 immediately gives rise to the following consequences:
\begin{enumerate}
    \item The induced map
    \[
    \X \ni x \stackrel{H}{\longmapsto} H(x) = (h_1(x),h_2(x),\dots) \in [-1,1]^\N
    \]
    is 1-1 and continuous, but in general it is not a homeomorphism of $\X$ onto a subspace of $[-1,1]^\N$, i.e. in general it fails to be an embedding.
    \item $H$ defines another topology $\tau_H$ on $\X$ which is weaker than $\tau$, i.e. $\tau \supset \tau_H$. This last topology is metrizable though, and hence both $\tau_H$ and $\tau$ are Hausdorff.
    \item By the well-known minimal property of compact topologies, both topologies coincide on $\tau$-compact sets $\mathcal{K}\subset\X$, and hence $\tau$-compact sets are metrizable.
    \item For any $\mathcal{K}\subset \X$ $\tau$-compact
    \[
        H\big |_{\mathcal{K}}:\mathcal{K}\to H(\mathcal{K})
    \]
    is a homeomorphism. Therefore, $H(\mathcal{K})$ is compact in $[-1,1]^\N$. 
    Besides, $\mathcal{K}$ is compact if and only if it is sequentially compact.
    \item For all $E\subset\X$ $\sigma$-compact subspaces of $(\X,\tau)$
    \[
        H\big |_{E}:E\to H(E)
    \]
    is a measurable isomorphism.
    \item $(\X,\tau)$ is functional Hausdorff, but need not be regular.
    \item If $A\subset\X$ is non-empty and if $\tau_A=\tau\cap A$ denotes its relative topology, then $(A,\tau_A,(h_i)_{i=1}^\infty)$ is also quasi-Polish.
\end{enumerate}

Before presenting many examples of quasi-Polish spaces, we make a few remarks, which will be tacitly used later: suppose we are given a countable family $\Bar{h}_i:\X\to \R,\,i\in\N$ of $\tau$-continuous functions which separate points of $\X$ but whose range is not necessarily in the interval $[-1,1]$. We can in any case exhibit a countable family $h_i,\,i\in\N$ satisfying Assumption \ref{assumption} by simply composing $\Bar{h}_i$ with a continuous and 1-1 function $\varphi:\R\to [-1,1]$, i.e. 
\[
h_i(x) := \varphi(\Bar{h}_i(x)),\quad x\in \X,
\]
like for example $\varphi(t)=\frac{2}{\pi}\arctan(t),\,t\in\R$.

More importantly, as it will become evident later, it is more convenient for our purposes to consider a homeomorphic version of the cube $[-1,1]^\N$, namely the Hilbert cube $\mathcal Q$ 
which can be seen as a subset of the separable Hilbert space $V=\ell^2(\N)$, i.e.
\[
\mathcal{Q}=\{ a\in V; \, 0\leq a_i\leq 1/i, \,i\in\N
\}.
\]
It can be easily shown that with the topology induced by the metric of $V$, $\mathcal{Q}$ is indeed homeomorphic to the cube $[-1,1]^\N$. In particular $\mathcal{Q}$ is a compact subset of $V$.

Therefore, given an arbitrary quasi-Polish space $(\X,\tau,(h_i)_{i=1}^\infty)$, by defining 
\[
f_i(x):= \frac{1}{2i}(h_i(x)+1),\quad i\in\N,\, x\in\X,
\]
we still obtain a separating sequence for which now $f_i(\X)\subset [0,1/i],\,i\in\N$.
For this reason, we may use from the beginning $(f_i)_{i=1}^\infty$ as a separating sequence; in view of the identification above, we now define the 1-1 and continuous map
\[
F: \X \to \mathcal{Q}\subset V, \quad
\X \ni x \stackrel{F}{\mapsto} F(x):= (f_i(x))_{i=1}^{\infty},
\]
and observe that clearly this map retains all the topological properties of the map $H$ above. In particular, for any $\mathcal{K}\subset \X$ $\tau$-compact 
\[
    F\big |_{\mathcal{K}}:\mathcal{K}\to F(\mathcal{K})
\]
is a homeomorphism and $F(\mathcal{K})$ is compact in  $V$.

In the rest of the paper, we will always resort to $\mathcal{Q}$ rather than $[-1,1]^\N$ and use the associated injection map $F$.

\subsection{Examples of quasi-Polish spaces} Let us now see several examples of quasi-Polish spaces, which will convince the reader of the width of this class. Since for our purposes it is important not only to prove that a topological space $(\X,\tau)$ is quasi-Polish but also to provide an ``amenable'' separating sequence $(h_i)_{i=1}^\infty$ (or $(f_i)_{i=1}^\infty$), many examples below are partially overlapping.

\begin{example}[Separable metrizable spaces]\label{ex: Separable metrizable spaces} Consider a topological space $(\X,\tau)$ which we assume to be separable and metrizable: let then $\rho$ be a metric on $\X$ inducing $\tau$; let $D=\{d_i;\,i\in\N\}\subset\X$ be dense and define
\[
h_i:\X\to\R,\quad x\mapsto h_i(x):= \rho(x,d_i).
\]
These functions are clearly $\tau$-continuous; moreover, given $x,x'\in \X$ with $x\neq x'$, there always exists an $i_0\in\N$ such that
\[
\rho(x,d_{i_0}) \neq \rho(x',d_{i_0}).
\]
Indeed, if for all $i\in\N$ we had $\rho(x,d_i) = \rho(x',d_i)$, then, by taking a sub-sequence $(i_n)_n$ such that $d_{i_n}\to x$ as $n\to\infty$ (which exists by density), we would get
\begin{equation*}
    \begin{split}
        \rho(x,d_{i_n}) &= \rho(x',d_{i_n}),\\
        \lim_{n}\rho(x,d_{i_n}) &= \lim_{n}\rho(x',d_{i_n}),\\
        0 &= \rho(x',x),\\
    \end{split}
\end{equation*}
i.e. $x=x'$. We conclude then that $(h_i)_{i=1}^\infty$ is a separating family for $(\X,\tau)$. This example shows, as anticipated above, that in particular Polish spaces (i.e. separable completely metrizable topological spaces) are quasi-Polish.
\end{example}

\begin{example}[Fr\'echet spaces carrying a Schauder basis]\label{ex: Frechet spaces carrying a Schauder basis} 
We consider now a Fr\'echet space $(\X,\tau)$, whose topology is induced by an increasing sequence $\{p_n;\,n\in\N\}$ of seminorms on $\X$, and with associated metric 
\begin{equation*}
    \rho(x,y):= \sum_{n=1}^\infty 2^{-n}\frac{p_n(x-y)}{1+p_n(x-y)},\quad x,y\in\X.
\end{equation*}
We assume that $\X$ carries a Schauder basis $(s_k)_{k=1}^\infty\subset \X$, and let $(\beta_k^\X)_{k=1}^\infty$ be the associated sequence of continuous linear projectors. We notice that not all Fr\'echet spaces carry a Schauder basis, and, if it is the case, then the space $(\X,\rho)$ is separable, a fact that would make this example fall into Example \ref{ex: Separable metrizable spaces}. Nonetheless, in the present setting we can provide a different separating sequence $(h_i)_{i=1}^\infty$ which may be relevant for applications, by leveraging the Schauder basis $(s_k)_{k=1}^\infty$. Indeed, from the very definition of the continuous linear functionals $\beta_k^\X$ we see that, given $x,y\in\X$ with $x\neq y$ there must exist an index $k_0$ for which $\beta_{k_0}^\X(x)\neq \beta_{k_0}^\X(y)$, and thus $(\beta_k^\X)_{k=1}^\infty$ supplies a separating sequence. Besides, we immediately see that the present argument actually holds more generally if $\X$ is only assumed to be a topological vector space carrying a Schauder basis. (See also Example \ref{ex: Separable normed space with the weak topology} below).

\end{example}

\begin{example}[Separable normed spaces]\label{ex: Separable normed spaces} This example too would fit into the framework on separable metrizable spaces (Example \ref{ex: Separable metrizable spaces}). Nonetheless, it is also instructive to consider it here. Let $(\X,\norm{\cdot}_\X)$ be 
a separable normed space (with topological dual $\X'$). Then in this case a separating sequence of functions can be provided in the following way. Let 
$\{\phi_i\}_{i\in\N}\subset \X'$ be such that
\begin{equation*}
	\norm{x}_\X=\sup_i \, \langle \phi_i,x\rangle,
	\quad x\in \X. 
\end{equation*}
Given $x_1\neq x_2$, we choose an integer $m$ such that 
$\langle \phi_m,x_1-x_2\rangle > \frac12\norm{x_1-x_2}_\X>0$. Hence 
$\langle\phi_m,x_1\rangle>\langle\phi_m,x_2\rangle$ and we conclude.
\end{example}

\begin{example}[Separable normed space with the weak topology]\label{ex: Separable normed space with the weak topology} Let $(\X,\norm{\cdot}_\X)$ be  again a separable normed space. Let us endow $\X$ with its weak topology $\sigma(\X,\X')$, i.e. the coarsest topology on $\X$ which makes all the elements of $\X'$ continuous. The last example then shows that $(\X,\sigma(\X,\X'),(\phi_i)_{i=1}^\infty)$ is also a quasi-Polish space. We would like to recall here that when $\X$ has infinite dimension, $(\X,\sigma(\X,\X'))$ is never metrizable: therefore, this constitutes a first example of a non-metric space that admits a sequence of functions that satisfies Assumption \ref{assumption}.

\end{example}

\begin{example}[Weak-star topology]\label{ex: Weak star topology} 
If $(\X,\norm{\cdot}_\X)$ is a separable normed space, then its dual $\X'$ endowed with the weak-star topology $\sigma(\X',\X)$ is quasi-Polish. To see this, take an arbitrary countable dense subset $D\subset \X, D=\{d_1,d_2,\dots\}$. Given $\phi_1,\phi_2\in \X', \phi_1\neq \phi_2$, there 
must exist $d_n\in D$ such that $\phi_1(d_n)\neq \phi_2(d_n)$, 
because, if this were not the case, then $\phi_1(d)=\phi_2(d)$, $d\in D$, and thus $\phi_1\equiv \phi_2$. Define $h_n:\X'\to\R$, $\phi\mapsto\phi(d_n)$, $n\in \N$. 
Then $h_n$ is $\sigma(\X',\X)$-continuous, and we conclude that $(\X',\sigma(\X',\X))$ is quasi-Polish, 
with a separating sequence provided by $\{h_n\}_{n\in\N}$. As a bonus, we also see that $\X'$ endowed with its strong topology is also a quasi-Polish space.

An example of this is the space of signed Radon measure (with a separable pre-dual), equipped with the weak-star topology: more precisely, given a locally compact second countable topological Hausdorff space $(Z,\tau_Z)$, then $\X= C_0(Z)$ equipped with the uniform norm is separable. (Indeed by \cite[Thm 5.3]{kechris}, we have that its one-point compactification $\Hat{Z}$ is metrizable, and hence second countable. Thus, $C(\Hat{Z})$ with the uniform norm is separable, and hence also its subspace $C_0(Z)$ must be separable.) Therefore the dual of $\X$ equipped with either the weak-star topology or with the strong topology is quasi-Polish. 

Another example is given by $L^\infty(\Omega, \mathcal{A},\mu)$ equipped with either the weak-star topology or the strong one, where $(\Omega,\mathcal{A},\mu)$ is a separable\footnote{We recall that the concept of being separable means that $\mathcal{A}$ is separable when viewed as a metric space with metric $\rho(A_1,A_2):=\mu[A_1 \triangle A_2]$. Any sigma-algebra generated by a countable collection of sets is separable, but the converse need not hold.} measure space and $\mu$ is assumed $\sigma$-finite.

We remark that also in the present case the whole space $(\X',\sigma(\X',\X))$ is never metrizable if $\X$ is infinite dimensional. More in general, if $(\X,\tau)$ is an arbitrary separable topological vector space, the same argument above shows that $(\X',\sigma(\X',\X))$ is a quasi-Polish space.

\end{example}

\begin{example}[Countable Cartesian product]\label{ex: Countable Cartesian product}
Given a collection of quasi-Polish spaces $(\X_m,\tau_m,(h_i^{(m)})_{i=1}^\infty)$ with $m\in\N$, it is straightforward to see that their topological product $\X:= \prod_m\X_m$ is quasi-Polish. Indeed, by denoting $\pi_m:\X\to \X_m$ the $m$-th canonical projection, the countable family $\{h^{(m)}_i\circ \pi_m;\, m\in\N, i\in\N\}$ provides a separating sequence.

\end{example}

\begin{example}[Space of bounded linear operators]\label{ex: Space of bounded linear operators}
Given two separable Banach spaces $(E,\norm{\cdot}_E)$ and $(B,\norm{\cdot}_B)$, let $\X=\mathcal L(E,B)$ be the Banach space of bounded linear operators from $E$ to $B$, endowed with the operator norm $\norm{\cdot}_{op}$. This space in general is non-separable. However, also in the present case, we can show that $\X$ is quasi-Polish. Let $(e_n)_n$ and $(b_m)_m$ be two dense sequences for $E$ and $B$ respectively. Define
\[
h_{n,m}:\X\to\R,\quad L\stackrel{h_{n,m}}{\longmapsto}
 \norm{Le_n - b_m}_B,\quad n,m\in\N.
\]
By means of the reverse triangle inequality, it is immediate to see that
\[
\abs{h_{n,m}(L) - h_{n,m}(\tilde{L})} \le \norm{Le_n - b_m - \tilde{L}e_n + b_m}_B \le \norm{L-\tilde{L}}_{op}\norm{e_n}_E,\quad L,\tilde{L}\in\X,
\]
showing that the maps $h_{n,m}$ are Lipschitz. Besides, suppose that $h_{n,m}(L)=h_{n,m}(\tilde{L})$ for all $n$ and $m$. By Example \ref{ex: Separable metrizable spaces}, we infer that $Le_n=\tilde{L}e_n$ for all $n$, and thus by continuity that $L=\tilde{L}$, i.e. the functions $h_{n,m}$ separate points, turning $\X$ into a quasi-Polish space.

Alternatively, by means of Example \ref{ex: Separable normed spaces}, another separating sequence is provided by
\[
h_{n,m}:\X\to \R,\quad L\stackrel{h_{n,m}}{\longmapsto}
 \langle \phi_m,Le_n\rangle ,\quad n,m\in\N,
\]
where $\{\phi_m\}_{m\in\N}\subset B'$ is a norm-replicating sequence.

\end{example}

\begin{example}[$C_b(\R^d)$]\label{ex: Cb}
 Consider $C_b(\R^d)$, namely the space of bounded and continuous functions on $\R^d$, endowed with its natural norm $\norm{\cdot}_\infty$. It is very well-known that it is not separable, and therefore not a Polish space. Nevertheless, since an element $u\in C_b(\R^d)$ is fully specified by the values it takes on $\mathbb Q^d$, the family of linear and continuous functionals $\{\delta_q;\, q\in\mathbb Q^d\}$ provides us with a separating sequence, making $(C_b(\R^d),\norm{\cdot}_\infty)$ a quasi-Polish space.
    
\end{example}

\begin{example}[$L^\infty(\Omega,\mathcal{A},\mu)$]\label{ex: L infinity}
The following example has been already discussed in Example \ref{ex: Weak star topology}; nevertheless, it has the merit of providing a more explicit construction of a separating sequence. Consider a measure space $(\Omega,\mathcal{A},\mu)$ such that
\begin{enumerate}
    \item $\mu$ is $\sigma$-finite, and therefore there exist $C_1\subset C_2\subset \cdots \subset C_n \subset \cdots,$ with $C_n\in\mathcal{A}$, $\mu[C_n]<\infty$ for all $n$ and $\Omega = \bigcup_{n\in\N}C_n$;
    \item $\mathcal{A}=\sigma(\mathcal E)$, where $\mathcal{E}$ is a countable $\pi$-system such that $\Omega$ is a finite or countable union of elements of $\mathcal{E}$.
\end{enumerate}
Let $\X:=L^\infty(\Omega,\mathcal{A},\mu)$ be endowed with its natural norm: as recalled above, in general $\X$ is not separable, and therefore it cannot be Polish. However, it is a quasi-Polish space, as the following computations will reveal. For $E\in\mathcal{E}$ and $n\in \N$ we consider the continuous (linear) functions
\[
\chi_{E,n}: \X \to \R,\quad u \mapsto \int_E I_{C_n}(\omega)u(\omega)\,\mu(d\omega).
\]
Suppose that there exist $u,v\in \X$ with $u\neq v$ such that $\chi_{E,n}(u)=\chi_{E,n}(v)$ for any $E\in\mathcal{E}$ and $n\in\N$. For fixed $n$, since $\mathcal{E}$ is a $\pi$-system such that $\Omega$ is a finite or countable union of elements of $\mathcal{E}$, a standard result from measure theory (see e.g. Thm 16.10 \cite{Billingsley}) implies that there exists $A_n\in\mathcal{A}$ such that
\[
\mu[A_n^c] =0, \quad I_{C_n}(\omega)u(\omega)=I_{C_n}(\omega)v(\omega)\quad \text{for } \omega\in A_n.
\]
Define the full measure set $A:=\bigcap_{n\in\N}A_n\in\mathcal{A}$. Given $\omega\in A$, there exists $n\in\N$ such that $\omega\in C_n$, namely $I_{C_n}(\omega)=1$. But $\omega\in A_n$ as well, and so
\[
I_{C_n}(\omega)u(\omega)=I_{C_n}(\omega)v(\omega),
\]
i.e. $u(\omega)=v(\omega)$, namely $u=v$ $\mu$-a.e., contradicting the fact that $u\neq v$. We conclude that the family $\{\chi_{E,n}; E\in\mathcal{E},n\in\N\}$ separates points; clearly, it is countable, and thus $\X$ is a quasi-Polish space. A particular case is when $(\Omega,\mathcal{A},\mu)=(\R^d,\mathcal{B}(\R^d),\mathcal{L}^d)$, where $\mathcal{L}^d$ is the $d$-dimensional Lebesgue measure. 
\end{example}

\begin{example}[BV functions]\label{ex: BV functions} 
This example easily follows from the previous one. We recall here that given an open subset $\Omega\subset \R^d$, the space of bounded variation functions $BV(\Omega)$ is made out of all the functions $u\in L^1(\Omega)$ for which there exists a finite vector Radon measure $Du:\Omega\to\R^d$ such that 
\[
\int_\Omega u(x) \operatorname{div} \phi(x)\,dx = -\int_\Omega \langle\phi(x),Du(x)\rangle,\quad \phi\in C^1_c(\Omega,\R^d).
\]
This space, endowed with the norm 
\[
\norm{u}_{BV(\Omega)} := \norm{u}_{L^1(\Omega)} + V(u,\Omega),\quad u\in BV(\Omega),
\]
where $V(u,\Omega):=\sup\{\int_\Omega u(x) \operatorname{div} \phi(x)\,dx;\, \phi\in C^1_c(\Omega,\R^d),\,\norm{\phi}_\infty\le 1
\}$, becomes a Banach space, which is not separable. Nonetheless, leveraging the previous example, we can define once more the functionals
\[
\chi_{E}: BV(\Omega) \to \R,\quad u \mapsto \int_E u(x)\,dx,\quad E\in\mathcal{E},
\]
where $\mathcal{E}$ is a countable $\pi$-system such that $\Omega$ is a countable union of elements of $\mathcal{E}$, and such that $\mathcal{B}(\Omega)=\sigma(\mathcal E)$. Clearly, these functionals are continuous and, arguing exactly as before, they separate points, making $BV(\Omega)$ a quasi-Polish space.
    
\end{example}

\begin{example}[Locally convex spaces with $(\X',\sigma(\X',X))$ separable]\label{ex: lcs with dual separable in the weak star}
Let $(\X,\tau)$ a locally convex space, such that 
\begin{equation}\label{eq: (H)}
    (\X',\sigma(\X',\X))\quad\text{is separable}.
\end{equation}
Then, we claim that $(\X,\tau)$ is quasi-Polish. To see that, let $(\phi_n)_n\subset \X'$ be a dense set (wrt $\sigma(\X',\X)$). We claim that $(\phi_n)_n$ separates points of $\X$. Suppose on the contrary that there exist $x_0\in\X, x_0\neq 0$ such that
\[
\phi_n(x_0)=0,\quad\text{for all } n\in\N.
\]
Fix an arbitrary $\ell\in\X'$ and consider the family of open neighborhoods of $\ell$ given by
\[
\mathcal{U}_\varepsilon(\ell)
=\{
\psi\in\X'; \, \abs{\langle \psi-\ell,x_0\rangle } < \varepsilon
\},\quad \varepsilon>0.
\]
By density of $(\phi_n)_n$, for each of these neighborhoods there must exists $\phi_{n_\varepsilon}$ such that $\phi_{n_\varepsilon}\in \mathcal{U}_\varepsilon(\ell)$, i.e. $\abs{\langle \phi_{n_\varepsilon} - \ell,x_0\rangle}<\varepsilon$, and thus
\[
\abs{\langle\ell,x_0\rangle}<\varepsilon.
\]
We conclude that $\ell(x_0)=0$. But since $\ell$ was arbitrary and the space $\X$ is locally convex, we conclude $x_0=0$, which is a contradiction.

\end{example}

\begin{example}[H\"older spaces] Let $(S,d_S)$ be a compact metric space, and $0<\alpha\le 1$. Let $\X=C^\alpha(S)$ be the space of $\alpha$-H\"older functions endowed with its natural norm, i.e.
\[
\norm{u}_\alpha = \sup_{s\in S}\abs{u(s)} + \sup_{s\neq t}\frac{\abs{u(s)-u(t)}}{d_S(s,t)^\alpha},\quad u\in C^\alpha(S).
\]
It is very well-known that $\X$ is not separable in general. Nonetheless, it is quasi-Polish. Indeed, $S$ is separable, and therefore we can find a numerable dense subset $D=\{s_1,s_2,\dots\}$. Define for $n\in \N$
\[
\delta_{s_n}: \X \to \R,\quad u\mapsto \langle \delta_{s_n},u\rangle = u(s_n).
\]
Clearly, $\delta_{s_n}$ is continuous, because if $\norm{u_k-u}_\alpha\to 0$, then $\sup_{s\in S}\abs{u_k(s)-u(s)}\to 0$. Besides, the family $(\delta_{s_n})_{n\in\N}$ separates points, because if there existed $u_1,u_2\in\X$ with $u_1\neq u_2$ and such that $\langle\delta_{s_n},u_1\rangle=\langle\delta_{s_n},u_2\rangle$ for each $n$, then (since continuous functions are determined by the values they take on a dense subset) it would follow that for any $s\in S$
\[
u_1(s)=\lim_nu_1(s_{n_k}) = \lim_nu_2(s_{n_k})= u(s)
\]
where $s_{n_k}\to s$ and $s_{n_k}\in D$, i.e. $u_1=u_2$, which is a contradiction.

We can expand this example further: consider again the compact metric space $(S,d_S)$ with designated origin $0\in S,\, 0\notin D $, and now a Banach space $(Z,\norm{\cdot}_Z)$ dual of some Banach space $(W,\norm{\cdot}_W)$. $Z$ is endowed with the weak-star topology $\sigma(Z,W)$.
Define 
\[
\X=C^\alpha(S,Z):=\{x: S\to Z;\, \text{continuous and s.t. } \norm{x}_{C^\alpha(S,Z)}<\infty
\},
\]
where
\[
\norm{x}_{C^\alpha(S,Z)}:=\norm{x(0)}_Z + \sup_{s\neq t}\frac{\norm{x(s)-x(t)}_Z}{d_S(s,t)^\alpha}.
\]
Then $(\X,\norm{\cdot}_{C^\alpha(S,Z)})$ is a Banach space: see e.g. \cite{CuchieroSchmockerTeichmann}. Assume now that $W$ is separable, and let $\{v_m\}_m\subset W$ be a dense subset. Define
\[
h_{n,m}:\X\to\R,\quad x\mapsto\langle x(s_n),v_m\rangle,\quad n,m\in\N.
\]
These maps are continuous, because if $x_k\to 0$ in $\X$, then in particular $\norm{x_k(0)}_Z\to 0$ and $\norm{x_k(s_n)-x_k(0)}_Z\to 0$, and thus $\norm{x_k(s_n)}_Z\to 0$ and $\langle x_k(s_n),v_m\rangle\to 0$. Besides, they separate points of $\X$: indeed, assume that there exist $x\neq y$ in $\X$ such that for all $m,n$ it holds
\[
\langle x(s_n),v_m\rangle = \langle y(s_n),v_m\rangle.
\]
The real-valued functions $S\ni s\mapsto \langle x(s),v_m\rangle$ and $S\ni s\mapsto \langle y(s),v_m\rangle$ are continuous by construction, and therefore, by density we infer that 
\[
\langle x(s),v_m\rangle = \langle y(s),v_m\rangle,\quad s\in S,m\in\N.
\]
By density again, we conclude $x=y$, which contradicts $x\neq y$. Therefore, $\X$ is a quasi-Polish space.

Approximation theorems and neural architectures on the non-separable space $C^\alpha(S)$ have already been considered by \cite[page 5]{CuchieroSchmockerTeichmann}. However, the numerical implementation of those architectures looks definitely much more complicated and less efficient than in the present case (refer to see Theorems \ref{thm: first} and \ref{thm: second}), because it requires a second layer of approximation needed to compute integrals of this kind
\[
\int_Sx(s)\nu(ds),\quad x\in C^\alpha(S)
\]
where $\nu$ is an arbitrary finite signed Radon measure on $S$. On the other hand, in the present setting, one essentially needs only to evaluate the input $x$ on a finite subset of the points $s_n$ of the dense set $D$: see Theorems \ref{thm: first} and \ref{thm: second} for more precise details.

\end{example}

\begin{example}[Spaces with ``the approximation of the identity'' property]\label{ex: approx identity property} Let $(\X,\tau)$ be a topological Hausdorff vector space. Suppose to have a sequence $(T_N)_{N\in\N}$ of continuous operators (not necessarily linear) $T_N:\X\to\X$ such that
\begin{itemize}
    \item $T_N(\X)\subset \X_N$, where $\X_N$ is an $N$-dimensional vector subspace,
    \item $T_N(x)\to x$ as $N\to\infty$ for every $x\in\X$.
\end{itemize}
Observe that we are not imposing neither that $\X_N\subset\X_{N+1}$ nor that the convergence is uniform on the compacts of $\X$, nor that the sequence $(T_N)_{N\in\N}$ is equicontinuous (as one usually does in the context of the bounded approximation property for locally convex spaces). Nonetheless, it seems to us that these conditions are the minimal requirements which one should impose in the context of numerical analysis and constructive mathematics in a broad sense, because they allow to replace an infinite-dimensional input $x\in\X$ with an approximation $x_N\simeq x$ that lives in an $N$-dimensional space.

We claim that such a space $\X$ is quasi-Polish. Indeed, for each $N\in\N$, let $\X_N \stackrel{\Phi_N}{\longrightarrow}\mathbb R^N$ be a linear isomorphism with $\Phi_N:=(\Phi_N^{(1)},\dots, \Phi_N^{(N)})$. Define 
\[
L_N:\X\to\R^N,\quad x\mapsto \Phi_N\circ T_N(x),
\]
and consider accordingly the countable family 
\[
\mathcal{S}:=\{
\Phi_N^{(i)}\circ T_N:\X\to\R\,;\, N\in\N,\, i=1,\dots N\}.
\]
These maps are continuous by composition. Suppose now that there exist $x,y\in\X,\,x\neq y$ such that $\Phi_N^{(i)}\circ T_N(x)=\Phi_N^{(i)}\circ T_N(y)$ for all $N\in\N,\, i=1,\dots N$. This clearly entails that $L_N(x)=L_N(y)$ for all $N$, and thus
\[
T_N(x) = T_N(y), \quad N\in\N,
\]
and so in the limit for $N\to\infty$ we obtain $x=y$, which is a contradiction. Therefore, $\mathcal{S}$ is a countable separating family, and $(\X,\tau)$ is a quasi-Polish space.

\end{example}

We conclude this subsection with a final example which highlights that in certain circumstances we can find an embedding $F$ preserving some desirable algebraic properties, as convexity of the subset where we wish to approximate our functions. This is a nice thing thing to have. Here it is how it works,

\begin{example}\label{ex: convexity}
Let $\X$ be a vector space such that there exists a sequence $(\bar{f}_n)_n$ of linear functionals which separate points. Set $\tau:=\tau(\bar{f}_n;\, n\in\N)$ be the coarsest topology which make these linear functionals continuous. In this way, $(\X,\tau)$ becomes a (locally convex) topological vector space which is quasi-Polish. We consider $\mathcal K\subset\Y$ compact and convex. We claim that there exists $f_n:\X\to [0,1/n],\,n\in\N$ continuous and separating points such that, by setting $F=(f_n)_{n=1}^\infty$, it holds that
\[
F(\mathcal{K})  \subset\mathcal{Q}\subset V
\]
is convex.

Indeed, since $\mathcal K$ is convex, it is also connected. Therefore, for each $n\in\N $, $\bar{f}_n(\mathcal K)=[\alpha_n,\beta_n]$ with $-\infty<\alpha_n<\beta_n<\infty$. Let $\delta_n(t):= \frac{t-\beta_n}{\beta_n-\alpha_n} + \frac12,\,t\in\R$. In that way, $$\delta_n([\alpha_n,\beta_n])=[-1/2,1/2].$$
Let $\chi:\R\to[-1,1]$ be continuous, strictly increasing and equal to the identity on $[-1/2,1/2]$. Observe, that trivially, if $0<t<1$ and $\abs{\xi}\le 1/2$, then $\chi(t\xi)=t\xi$. Finally, define for $n\in\N$,
\[
f_n:\X\to [0,1/n],\quad f_n(y):= \frac{1}{2n}(\chi(\delta_n(\Bar{f}_n(y)))+1),
\]
which is still a continuous separating sequence. Let $a,b\in F(\mathcal{K})$ with $a\neq b$ to avoid trivialities. Then there exist unique $y_a,y_b\in\mathcal K$ with $y_a\neq y_b$ and $F(y_a)=a,F(y_b)=b$. Let $0<t_a,t_b<1,\,t_a+t_b=1$. Set $c:=t_aa+t_bb=t_aF(y_a)+t_bF(y_b)\in  V$.

For any given $n$ it holds
\[
\begin{split}
    t_af_n(y_a)&= \frac{1}{2n}(t_a\chi(\delta_n(\bar{f}_n(y_a)))+t_a)\\
    &=\frac{1}{2n}(\chi(t_a\delta_n(\bar{f}_n(y_a)))+t_a)\\
    &=\frac{1}{2n}\left(\chi\left(t_a\frac{\bar{f}_n(y_a)-\beta_n}{\beta_n-\alpha_n}+\frac{t_a}{2}\right)+t_a\right).
\end{split}
\]
Because $0<t_a<1$ and $\frac{\bar{f}_n(y_a)-\beta_n}{\beta_n-\alpha_n}\in[-1,0]$, it follows that $\abs{t_a\frac{\bar{f}_n(y_a)-\beta_n}{\beta_n-\alpha_n}+\frac{t_a}{2}}\leq 1/2$ and hence
\[
t_af_n(y_a) = \frac{1}{2n}\left[\left(t_a\frac{\bar{f}_n(y_a)-\beta_n}{\beta_n-\alpha_n} + \frac{t_a}{2}
\right)+t_a
\right],
\]
and, similarly, $t_bf_n(y_b) = \frac{1}{2n}\left[\left(t_b\frac{\bar{f}_n(y_b)-\beta_n}{\beta_n-\alpha_n} + \frac{t_b}{2}
\right)+t_b
\right]$. Hence,
\[
\begin{split}
t_af_n(y_a)+t_bf_n(y_b)&= \frac{1}{2n}\left[t_a\frac{\bar{f}_n(y_a)-\beta_n}{\beta_n-\alpha_n} + t_b\frac{\bar{f}_n(y_b)-\beta_n}{\beta_n-\alpha_n}
+ \frac{1}{2}
+1
\right]\\
&= \frac{1}{2n}\left[
\frac{\bar{f}_n(t_ay_a+t_by_b)-\beta_n}{\beta_n-\alpha_n}
+\frac12 + 1\right]\\
&=\frac{1}{2n}[\delta_n(\bar{f}_n(t_ay_a+t_by_b))+1
].
\end{split}
\]
Observe that by convexity $t_ay_a+t_by_b\in\mathcal K$, and so $\bar{f}_n(t_ay_a+t_by_b)\in [\alpha_n,\beta_n]$ and therefore $\delta_n(\bar{f}_n(t_ay_a+t_by_b))\in [-1/2,1/2]$. We conclude
\[
\begin{split}
    t_af_n(y_a)+t_bf_n(y_b) &= \frac{1}{2n} [\chi(\delta_n(\bar{f}_n(t_ay_a+t_by_b)))+1]\\
    &= f_n(t_ay_a+t_by_b),\quad n\in \N,
\end{split}
\]
i.e. $c=t_aF(y_a)+t_bF(y_b)=F(t_ay_a+t_by_b)$ with $t_ay_a+t_by_b\in\mathcal K$. Namely, $c\in F(\mathcal K)$.

As a typical application thereof, consider a separable normed space $(\X,\norm{\cdot})$ with $D=\{d_1,d_2,\dots\}\subset \X$ dense subset. Let $$\bar{f}_n:(\X',\sigma(\X',\X))\to \R,\quad \phi \stackrel{\bar{f}_n}{\longmapsto} \langle\phi,d_n\rangle$$ and $\mathcal K\subset \X'$ compact and convex, e.g. $\mathcal{K}=\{\phi\in\X';\,\norm{\phi}_{\X'}\le 1\}$. More specifically, let $(Z,\tau_Z)$ be a locally compact second countable Hausdorff space, and consider $\X:=C_0(Z)$, the space of continuous functions on $Z$ vanishing at infinity. Then, once again, $\X'$ is the space of signed Radon measures, which we endow with $\sigma(\X',\X)$. Consider the $\sigma(\X',\X)$-closed subset 
\[
\mathcal P := \bigcap_{u\in C_0(Z)_+}\{\mu;\, \langle \mu,u\rangle\ge 0
\},
\]
i.e. the ``positive cone'', where clearly $C_0(Z)_+=\{u\in C_0(Z);\, u\ge 0\}$, and 
\[
\mathcal K:=\mathcal{P}\cap \{\mu\in\X';\,\norm{\mu}_{\X'}\le 1\},
\]
i.e. the set of all Radon sub-probability measures $\mu:\mathcal{B}(Z)\to [0,1]$: it is evidently $\sigma(\X',\X)$-compact and convex, and thus the argument above applies.
\end{example}

\subsection{Infinite-dimensional neural networks} \label{subsec: Infinite-dimensional neural networks}

We now describe the neural network architectures from \cite{BDG0} whose scope and domain of applications we are going to significantly expand, and briefly recall the main results we will need.

To this end, we first observe that, even though the architectures from \cite{BDG0} can deal with inputs from an arbitrary Fr\'echet space $E$, in the present context, given the geometric structure of the problem at hand, it will be enough for our purposes to focus only on the relatively easier case when $E$ is the Hilbert space $V=\ell^2(\N)$. As usual, we will write $\langle\cdot,\cdot\rangle$ for the duality in place between $  V'    $ and $V$, with no further specification if no possibility of confusion arise.

In order to define the infinite-dimensional analogue of a neuron, $a^{\top}x +b$ in \eqref{eq: finite-dim neurons} is replaced by an affine function on $V$, the activation function $\sigma : \mathbb{R} \rightarrow \mathbb{R}$ by a function in $C(V,V)$, and the scalar $L$ by a continuous linear form.
For $\phi\in   V'    , A\in \mathcal{L} (V), b\in V$ a neuron, $\mathcal{N}_{\phi,A,b}$ is then defined by 
$$\mathcal{N}_{\phi,A,b}(x)= \langle \phi ,\sigma (Ax +b)\rangle,\quad x\in V,$$
and a one layer neural network is a finite sum of neurons
\[  
\mathcal{N}(x) =
\sum_{j=1}^{J} \mathcal{N}_{\phi_j,A_j,b_j}(x), \quad x\in V.
\]
Compare \eqref{eq: finite-dim nn}. Then, analogously as above, one asks for conditions on $\sigma:V \rightarrow V$ that ensure that 
\[
\mathfrak N(\sigma):=\left\{ \sum_{j=1}^{J} \mathcal{N}_{\phi_j, A_j,b_j};\, 
J\in\N,\phi_j\in   V'    ,A_j\in \mathcal{L}(V) ,b_j \in V  \right\}
\] 
is dense in $C(V)$ under some suitable topology. 

In \cite{BDG0}, the following separating property for the activation function $\sigma$, which can be seen as the infinite-dimensional counterpart to the well known sigmoidal property for functions from $\R$ to $\R$ (see \cite{Cybenko1989}), was introduced:
\begin{definition}{Separating property:}\label{sigmoid}
There exist $0\neq \psi\in   V'    $ and $u_+,u_-,u_0\in V$ such that either $u_+ \notin \Span \{u_0,u_-\}$ or $u_- \notin \Span \{u_0,u_+ \}$ and such that 
\begin{equation}\label{eq: abstract condition on sigma}
\begin{cases}
\lim_{\lambda\to\infty} \sigma(\lambda x) = u_+, \text{ if } x\in \Psi_+\\
\lim_{\lambda\to\infty} \sigma(\lambda x) = u_-, \text{ if } x\in \Psi_-\\
\lim_{\lambda\to\infty} \sigma(\lambda x) = u_0, \text{ if } x\in \Psi_0\\
\end{cases}
\end{equation}
where we have set 
\begin{equation*}
    \Psi_+ =\{ x\in V;\, \langle\psi,x\rangle >0 \}, \quad \Psi_- =\{ x\in V;\, \langle\psi,x\rangle <0 \}
\end{equation*}
and $\Psi_0=\ker(\psi)$.
\end{definition}

We point out that as a very particular case of the separating property one may choose $u_0=u_-=0$ and $u_+\neq 0$ for instance. Then for any $\psi \in   V'    
$ and a function $\beta\in C(\R)$ with $\lim_{\xi\to\infty} \beta(\xi)=1,\;\lim_{\xi\to-\infty} \beta(\xi)=0$ and $\beta(0)=0$, we can define a separating $\sigma : V \rightarrow V$ by:
\begin{equation*}
    \sigma(x) = \beta(\psi(x)) u_+,\quad x\in V.
\end{equation*}

The following result is an immediate consequence of \cite[Thm. 2.3 and 2.8]{BDG0}, and shows the density of $\mathfrak N(\sigma)$ if the activation function $\sigma$ satisfies the separating property.

\begin{theorem}\label{thm: density}
Let $\sigma:V\to V$ be continuous, satisfying \eqref{eq: abstract condition on sigma} and with bounded range $\sigma(V)$. 
Then $ \mathfrak N(\sigma)$ is dense in $C(V)$ when equipped with the topology of uniform convergence on compacts. In other words, given $f\in C(V)$, then, for any compact subset $K$ of $V$, and any $\varepsilon>0$, there exists $ \sum_{j=1}^J \mathcal{N}_{\phi_j,A_j,b_j}\in  \mathfrak N(\sigma)$ with suitable $J\in\N, \phi_j\in   V'    ,A_j\in\mathcal{L}(V)$ and $b_j\in V$ such that
\begin{equation*}\label{approx:prop}
    \sup_{x\in K}\abs{f(x) - \sum_{j=1}^J  \mathcal{N}_{\phi_j,A_j,b_j}(x)}  < \varepsilon.
\end{equation*}
\end{theorem}

Moreover, the following result ensures that one can approximate a given infinite dimensional neural network as above arbitrary well via a neural network that is constructed from finite dimensional maps and which can thus be trained. The result is a very special case of \cite[Prop. 4.1]{BDG0}. 
\begin{proposition}\label{prop: finite dimensional approx, Banach}\label{finite:dim:projection:NN}
Let $(e_k)_{k\in\N}$ be an orthonormal basis for $V$. For each $N\in\N$ let
$$
 \Pi_N:  V \to \Span\{e_1,\dots ,e_N\}
$$
be the orthogonal projection on the first $N$ elements of the basis. Let $\sigma: V \to V $ be Lipschitz. Let $f\in C(V)$, $K\subset V$ compact and $\varepsilon>0$. Assume
\begin{equation*}
    \mathcal N^{\varepsilon} (x) = \sum_{j=1}^J\langle \phi_j,\sigma(A_jx+b_j)\rangle,\quad x\in V
\end{equation*}
with $\phi_j\in   V'    ,A_j\in\mathcal{L}(V)$ and $b_j\in V$ such that 
\begin{equation*}
    \sup_{x\in K}\abs{f(x)-\mathcal N^{\varepsilon}(x)}<\varepsilon.
\end{equation*}
Fix $\bar{\varepsilon}>0$. Then there exists $N_\ast=N_\ast(\mathcal N^\epsilon,\bar{\varepsilon})\in\N$ such that for $N\geq N_\ast$
\begin{equation}\label{approx:finite}
    \sup_{x\in K}\abs{f(x)-\sum_{j=1}^J\langle \phi_j\circ\Pi_N,\sigma(
    \Pi_{N}A_j\Pi_{N}x+\Pi_{N}b_j)\rangle}<\varepsilon+\Bar{\varepsilon}.
\end{equation}
\end{proposition}

We mention that the function $\mathcal{N}^{\varepsilon}: V \rightarrow \R$, which is required in the proposition above, exists for instance in view of Theorem \ref{thm: density}, as soon as one assumes additionally that $\sigma$ satisfies \eqref{eq: abstract condition on sigma} and has bounded range $\sigma(V)$. 
Observe that if $\sigma$ is Lipschitz continuous, then every $\mathcal{N}\in\mathfrak{N}(\sigma)$ is also Lipschitz continuous. 

We will profit from Theorem \ref{thm: density} and Proposition \ref{prop: finite dimensional approx, Banach} in the rest of the paper.

\section{Quasi-Polish neural networks architectures}\label{sec: Quasi-Polish neural networks architectures} 

In this section, we will rigorously introduce our neural network architectures. To this end, let $\sigma:V\to V$ be a continuous function. Based on this activation function, we then define the space $\mathfrak N(\sigma)$ of the \textbf{$V$-scalar neural networks} as the subspace of $C(V)$ 
\begin{equation}\label{eq: NN on V}
    \mathfrak N(\sigma):=\left\{ \sum_{j=1}^{J} \mathcal{N}_{\phi_j, A_j,b_j};\, 
J\in\N,\phi_j\in   V'    ,A_j\in \mathcal{L}(V) ,b_j \in V  \right\} \subset C(V)
\end{equation}
where $\mathcal{N}_{\phi_j,A_j,b_j}(x)= \langle \phi_j ,\sigma (A_jx +b_j)\rangle,\,x\in V$.

Given a quasi-Polish space $\X$ with associated injection map $F:\X\to\mathcal{Q}\subset V$, we define the space $\mathfrak{N}_{F,\sigma}(\X)$ of \textbf{quasi-Polish scalar neural networks} as 
\begin{equation}\label{eq: NN on X}
    \mathfrak{N}_{F,\sigma}(\X):= \left\{
    \sum_{j=1}^J \mathcal{N}_{\phi_j, A_j,b_j} \circ F;\, 
J\in\N,\phi_j\in   V'    ,A_j\in \mathcal{L}(V) ,b_j \in V 
    \right\} \subset C(\X),
\end{equation}
namely, we pre-compose the elements of $\mathfrak{N}(\sigma)$ with the injection map $F$, and thus a typical element $\mathcal{N}$ of $\mathfrak{N}_{F,\sigma}(\X)$ will look like
\[
\mathcal{N}(x) = \sum_{j=1}^J \langle \phi_j ,\sigma (A_jF(x) +b_j)\rangle,\quad x\in\X,
\]
with $\phi_j\in V',A_j\in\mathcal{L}(V)$ and $b_j\in V$. We remark that $\mathcal{N}\in C(\X)$ because $F:\X\to V$ is continuous.

Finally, given a topological vector space $(E,\tau_E)$, we define the space $\mathfrak{N}_{F,\sigma}(\X,E)$ of \textbf{quasi-Polish vector neural networks}
\begin{equation}\label{eq: NN on X to E}
 \mathfrak{N}_{F,\sigma}(\X,E):= \left\{ \sum_{m=1}^M\mathcal{N}^{(m)}(x)v^{(m)};\,M\in\N, \mathcal{N}^{(m)} \in \mathfrak{N}_{F,\sigma}(\X), v^{(m)}\in E    
    \right\} \subset C(\X)\otimes E , 
\end{equation}
and thus a typical element $\mathcal{N}$ of $\mathfrak{N}_{F,\sigma}(\X,E)$ will look like
\[
\mathcal{N}(x)=\sum_{m=1}^M\sum_{j=1}^{J_m}\langle \phi_j^{(m)},\sigma(A_j^{(m)}F(x) + b_j^{(m)} )
\rangle v^{(m)},\quad x\in\X,
\]
with $\phi_j^{(m)}\in V',A_j^{(m)}\in\mathcal{L}(V)$ and $b^{(m)}_j\in V$.

We will also need a ``finite-dimensional'' variant of these architectures. First of all, for $N\in\N$, we define the operators
\begin{equation}\label{eq: LambdaN}
\mathfrak{N}_{F,\sigma}(\X) \stackrel{\Lambda_N}{\longrightarrow}\mathfrak{N}_{F,\sigma}(\X)  
\end{equation}
in such a way: for a neuron $\mathcal{N}_{\phi,A,b}\circ F\in \mathfrak{N}_{F,\sigma}(\X)$ we set
\[
\Lambda_N(\mathcal{N}_{\phi,A,b}\circ F) = \langle \phi\circ\Pi_N ,\sigma (\Pi_N\circ A\circ\Pi_NF +\Pi_Nb)\rangle,
\]
and then we extend by linearity to the whole $\mathfrak{N}_{F,\sigma}(\X) $. Once again, the operators $\Pi_N$ are the orthogonal projection in $V$ onto $\Span\{e_1,\dots,e_N\}$ with respect to a pre-assigned orthornomal basis $(e_k)_{k\in\N}\subset V$. We observe that, since $\phi\circ\Pi_N\in V', \Pi_N\circ A\circ\Pi_N\in\mathcal{L}(V)$ and $\Pi_Nb\in V$, then indeed $\Lambda_N(\mathcal{N}_{\phi,A,b}\circ F)\in \mathfrak{N}_{F,\sigma}(\X)$, viz. the operator is well-defined. 

We set accordingly 
\begin{equation}\label{eq: finite NN on X}
    \mathfrak{N}_{F,\sigma,N}(\X) := \Lambda_N(\mathfrak{N}_{F,\sigma}(\X) ),\quad N\in\N, 
\end{equation}
and 
\begin{equation}\label{eq: finite NN on X to E}
   \mathfrak{N}_{F,\sigma,N}(\X,E):= \left\{ \sum_{m=1}^M\mathcal{N}^{(m)}(x)v^{(m)};\,M\in\N, \mathcal{N}^{(m)} \in
    \mathfrak{N}_{F,\sigma,N}(\X), v^{(m)}\in E    
    \right\},\quad N\in\N. 
\end{equation}

\section{Main results}\label{sec: Main results}

The following lemma simply re-states the metrizability of compact subsets of quasi-Polish spaces, and shows that the map $F$ induced by a separating sequence $(f_i)_{i\in\N}$ is an isometry between compact subsets of $\X$ and their images in $V$, a fact which will come handy in the following: compare e.g. Subsection \ref{subsec: Universal approximation results for targets that are quasi-Polish}.

\begin{lemma}\label{lemma: metrizability of compacts}
Let $(\X,\tau)$ be a quasi-Polish space, $(f_i)_{i\in\N}$ be a separating sequence, and let $F=(f_1,f_2,\dots)$. Let $\tau_F\subset\tau$ be the topology induced by $F$. Then $\tau_F$ is metrizable and, for any $\tau$-compact subset $\mathcal K\subset \X$, it holds that
\[
F|_\mathcal{K}:(\mathcal{K},\tau_\mathcal{K})\to F(\mathcal K)\subset V
\]
is an isometry, where $\tau_\mathcal{K}:=\tau\cap\mathcal{K}$. The same holds for $(F|_\mathcal{K})^{-1}:F(\mathcal{K})\to(\mathcal{K},\tau_\mathcal{K})$.
\end{lemma}
\begin{proof}
    Let us define 
    \[
    d_F:\X\times\X\to[0,\infty),\quad (x,y)\mapsto d_F(x,y)=\norm{F(x)-F(y)}_{V}.
    \]
    In view of the injectivity of $F$, it is immediate to see that $d_F$ is a metric on $\X$. Let us denote $\tau_{d_F}$ the topology induced by $d_F$. We want to show that $\tau_F=\tau_{d_F}$.

    Given an arbitrary net $<x_\alpha>_\alpha\subset \X$ converging to $x\in\X$ with respect to $\tau_F$, we have by definition that $F(x_\alpha)\to F(x)$ in $V$. Therefore, $d_F(x_\alpha,x)\to 0$, namely $x_\alpha\to x$ with respect to $\tau_{d_F}$, and so $\tau_F\supset\tau_{d_F}$. Let now be $<x_\alpha>_\alpha\subset \X$ a net converging to $x\in\X$ with respect to $\tau_{d_F}$: then $F(x_\alpha)\to F(x)$ in $V$, showing that $F$ is $\tau_{d_F}$-continuous. Hence, $\tau_{d_F}\supset\tau_F$.

    Since for $\tau$-compact subsets we know that $\tau_\mathcal{K}=\tau_F\cap\mathcal{K}$ and now also that $\tau_K=\tau_{d_F}\cap\mathcal K$, the statement that $F|_\mathcal{K}$ as well as its inverse are isometries is clear.
\end{proof}

\subsection{Universal approximation theorem for quasi-Polish spaces: the scalar case}

We are now going to state and prove our first main result.

\begin{theorem}\label{thm: first}
Let $(\X,\tau)$ be a quasi-Polish space with injection map $F$. Assume that the activation function $\sigma:V\to V$ is continuous, satisfying the separating property \eqref{eq: abstract condition on sigma} and such that $\sigma(V)$ is bounded. Then $\mathfrak{N}_{F,\sigma}(\X)$ is dense in $C(\X)$ with respect to the topology of uniform convergence on compacts, in the sense that for any $g:\X\to\R$ continuous, any $\mathcal{K}\subset\X$ compact and error $\varepsilon>0$, there exists $\mathcal{G}^\varepsilon\in\mathfrak{N}_{F,\sigma}(\X) $ 
\[
\mathcal{G}^\varepsilon(x)=\sum_{j=1}^J\langle \phi_j,\sigma(A_jF(x) + b_j)\rangle,\quad x\in \X
\]
for suitable $J,\in\N,\phi_j\in  V',A_j\in\mathcal{L}( V),b_j\in V$ such that
\[
\sup_{x\in\mathcal{K}}\abs{ \mathcal{G}^\varepsilon(x) -g(x)}<\varepsilon.
\]
\end{theorem}
\begin{proof}
Let $g:\X\to\R$ be a continuous function, and $\mathcal{K}\subset\X$ an arbitrary compact subset. Then $g\circ F^{-1}: F(\mathcal{K})\subset V\to\R$ is continuous and $F(\mathcal{K})$ is compact. Besides, $(g\circ F^{-1})(F(\mathcal{K}))\subset [a,b]$ for some $a < b$. Since $ V$ is a metric space, $F(\mathcal{K})$ is closed and we can apply Tietze's extension theorem to find $\mathcal{U}: V\to[a,b]$ continuous extension of the function $g\circ F^{-1}$.

By virtue of Theorem \ref{thm: density}, we can now approximate uniformly $\mathcal{U}$ on $F(\mathcal{K})$ via elements of $\mathfrak{N}(\sigma)$: thus, given $\varepsilon>0$, we can find $\mathcal{N}^\varepsilon\in\mathfrak{N}(\sigma)$ 
$$\mathcal{N}^\varepsilon(z) = \sum_{j=1}^J\langle \phi_j,\sigma(A_jz + b_j)\rangle,\quad z\in V$$
    for suitable $J\in\N,\phi_j\in  V',A_j\in\mathcal{L}( V),b_j\in V$ such that 
\[
    \abs{\mathcal{N}^\varepsilon(z) - \mathcal{U}(z)} < \varepsilon,\quad z\in F(\mathcal{K}),
    \]
i.e.
    \[
    \abs{\mathcal{N}^\varepsilon(z) - (g\circ F^{-1})(z)} < \varepsilon,\quad z\in F(\mathcal{K}).
    \]
Therefore, $\mathcal{G}^\varepsilon:= \mathcal{N}^\varepsilon \circ F\in\mathfrak{N}_{F,\sigma}(\X)$, and for $x\in \mathcal{K}$, it holds
\[
\abs{\mathcal{G}^\varepsilon(x)-g(x)} = \abs{\mathcal{N}^\varepsilon(F(x))-g(x)} =
\abs{\mathcal{N}^\varepsilon(F(x)) - (g\circ F^{-1})(F(x))} <\varepsilon,
\]
i.e. 
\[
\sup_{x\in\mathcal{K}}\abs{\sum_{j=1}^J\langle \phi_j,\sigma(A_jF(x) + b_j)\rangle -g(x)}<\varepsilon.
\]
\end{proof}

The second result is in the same spirit of Proposition \ref{prop: finite dimensional approx, Banach}:

\begin{theorem}\label{thm: second}
 Let $(\X,\tau)$ be a quasi-Polish space with injection map $F$. Assume that the activation function $\sigma:V\to V$ is Lipschitz. Let $(e_k)_{k\in\N}$ be an orthonormal basis for $V$, and for each $N\in\N$ consider the orthogonal projection on the first $N$ elements of this basis
$$
 \Pi_N:  V \to \Span\{e_1,\dots ,e_N\}.
$$
Let $g:\X\to\R$ continuous, $\mathcal{K}\subset\X$ compact and $\varepsilon>0$, and assume there exists $\mathcal{G}^\varepsilon\in\mathfrak{N}_{F,\sigma}(\X)$ 
\[
\mathcal{G}^\varepsilon(x)=\sum_{j=1}^J\langle \phi_j,\sigma(A_jF(x) + b_j)\rangle,\quad x\in\X
\]
for suitable $J\in\N,\phi_j\in  V',A_j\in\mathcal{L}( V),b_j\in V$ such that 
\[
\sup_{x\in\mathcal{K}}\abs{g(x)-\mathcal{G}^\varepsilon(x)}<\varepsilon.
\]
Fix $\bar{\varepsilon}>0$. Then there exists $N_\ast=N_\ast(\mathcal G^\epsilon,\bar{\varepsilon})\in\N$ such that for $N\geq N_\ast$
\begin{equation}\label{eq: second (2)}
\sup_{x\in\mathcal{K}}\abs{\Lambda_N(\mathcal{G}^\varepsilon)(x) - \mathcal{G}^\varepsilon(x)
}<\bar{\varepsilon}.
\end{equation}
and
\begin{equation}\label{eq: second}
\sup_{x\in\mathcal{K}}\abs{g(x)-\Lambda_N(\mathcal{G}^\varepsilon)(x)
}<\varepsilon + \bar{\varepsilon}.
\end{equation}
\end{theorem}
\begin{proof}
A careful inspection of the proof of \cite[Prop. 4.1]{BDG0} will reveal that, given a Lipschitz function $\sigma:V\to V$, a compact subset $\mathcal{K}_0$ of $V$ and a neural network architecture $\mathcal{N}_0\in\mathfrak N(\sigma)$, $\mathcal N_0(z)=\sum_{j=1}^J\langle \tilde{\phi}_j,\sigma(\tilde{A}_jz + \tilde{b}_j)\rangle$ then, for any $\bar{\varepsilon}>0$, there exists $N_0=N_0(\mathcal{N}_0,\bar{\varepsilon})\in\N$ such that for all $N\ge N_0$ we have
    \begin{equation}\label{eq: corollary of thm: second (1)}
\sup_{z\in\mathcal{K}_0}\abs{\mathcal N_0(z)-\sum_{j=1}^J\langle \tilde{\phi}_j\circ\Pi_N,\sigma(\Pi_N\circ \tilde{A}_j\circ \Pi_N z + \Pi_N \tilde{b}_j)\rangle }< \Bar{\varepsilon}.
    \end{equation}

In view of this, it is immediate to transport the result now on $\X$ via $F$, because given $\mathcal{G}^\varepsilon\in\mathfrak{N}_{F,\sigma}(\X)$ as in the assumptions, then
\[
\sum_{j=1}^J\langle \phi_j,\sigma(A_jz + b_j)\rangle,\quad z\in V
\]
belongs to $\mathfrak N(\sigma)$, and thus we can apply \eqref{eq: corollary of thm: second (1)} on $\mathcal{K}_0:=F(\mathcal K)$ and we obtain
\[
\sup_{x\in\mathcal{K}}\abs{\Lambda_N(\mathcal{G}^\varepsilon)(x) - \mathcal{G}^\varepsilon(x)
}<\bar{\varepsilon}
\]
for all $N\ge N_\ast$ where $N_\ast(\mathcal{G}^\varepsilon,\Bar{\varepsilon})$ is suitable. The rest is now obvious.  
\end{proof}

\begin{corollary}\label{corr: to second}
    Assume the same settings of Theorem \ref{thm: second} and in addition that the activation function $\sigma:V\to V$ satisfies the separating property \eqref{eq: abstract condition on sigma} and that $\sigma(V)$ is bounded. Then $\bigcup_{N\in\N} \mathfrak{N}_{F,\sigma,N}(\X) $ is dense in $C(\X)$ with respect to the topology of uniform convergence on compacts.
\end{corollary}

Let us comment on the last results. First of all, the ``weights'' $\phi_j\circ\Pi_N,\Pi_N\circ A_j\circ\Pi_N$ and $\Pi_N\circ b_j$ of the neural architecture $\Lambda_N(\mathcal{G}^\varepsilon)$ are finite-dimensional objects, and hence can now easily be programmed in a computer. We see that for large $N$, it is sufficient to consider the finite dimensional input values $\Pi_N (z)\in V$ instead of $z\in V$, and then successively the restriction of the operators $\Pi_N\circ A_j$ and $\phi_j$ to $\Span\{e_1,\dots ,e_N\}$ instead of the maps $A_j$ and $\phi_j$ for $j=1,\dots , J$. These linear maps $\Pi_N\circ A_j$ and $\phi_j$ are finite dimensional when restricted to $\Span\{e_1,\dots ,e_N\}$, and hence specified by a finite number of parameters. More precisely, the action of $\phi_j$ will be prescribed by the scalars $\phi_j(e_1),\dots,\phi_j(e_N)$ and the action of $\Pi_N\circ A_J\circ\Pi_N$ will be specified by $\left\{ (A_je_m,e_k)_V \right\}_{m,k=1}^N$. The architecture 
\[
\Lambda_N(\mathcal{G}^\varepsilon)(x)=\sum_{j=1}^J\langle \phi_j\circ\Pi_N,\sigma(\Pi_N\circ A_j\circ\Pi_NF(x) + \Pi_Nb_j)\rangle,\quad x\in\X
\]
thus resembles a classical neural network. However, instead of the typical one dimensional activation function, the function $\Pi_N \circ \sigma$ restricted to $\Span\{e_1,\dots ,e_N\}$ is multidimensional. 

Besides, if we now choose the canonical orthonormal basis of $V$, namely $$e_k:=(0,0,\dots, 0,1,0,\dots),\quad k\in \N$$ with entry equal to 1 at the $k$-th slot, then the input $F(x)\in V$ becomes
\[
\Pi_N F(x) = f_1(x)e_1 + \dots + f_N(x)e_N = (f_1(x),\dots,f_N(x),0,0,\dots). 
\]
We learn from here that under the setup of Theorem \ref{thm: second} one does not even need to specify the full separating sequence $(f_i)_{i=1}^\infty$, but that is enough to stop at $N$ sufficiently large. This fact has very deep consequences: let us assume now for instance that $\X$ is a separable Fr\'echet space. A major drawback of the architectures constructed in \cite{BDG0} and in \cite{GKL} is that one is required to find a Schauder basis for $\X$ in order to project down the input $x\in\X$ to a finite-dimensional subspace and construct a neural architecture specified by a finite number of parameters. However, this approach can be quite cumbersome in some applications, if not impossible at all. This is because, first, a Schauder basis cannot exist at all, as Enflo \cite{10.1007/BF02392270} famously showed in the 70s; second, even if it a basis exists, it might be difficult to find an explicit one. However, in the extremely more flexible present framework, since the Fr\'echet space $\X$ is assumed to be separable, it is enough to find a countable dense subset $D\subset\X$ to construct a separating sequence $(f_i)_{i=1}^\infty$ as demonstrated in Example \ref{ex: Separable metrizable spaces}. Besides, in some cases, separability is even superfluous, as Examples \ref{ex: Space of bounded linear operators}, \ref{ex: L infinity}, \ref{ex: BV functions}, and \ref{ex: lcs with dual separable in the weak star} indicate.

\subsection{Universal approximation results for quasi-Polish spaces: the vector-valued case}

In this part, we are going to treat neural network architectures defined on a quasi-Polish space $\X$ and taking values in a vector space $E$ belonging to some specific category of vector spaces. More precisely, we will examine first the class of locally convex spaces (see Proposition \ref{prop: the target is l.c.s}), and later on we will move on to the class of topological vector spaces (see Proposition \ref{prop: the target is t.v.s}), which is the most general class of spaces which combine a linear structure compatible with a topology. As a result, we will be able to approximate continuous functions $g$ from $\X$ to $E$ (=topological vector space) with respect to the topology of uniform convergence on compact subsets. While clearly Proposition \ref{prop: the target is l.c.s} can be seen as sub-case of Proposition \ref{prop: the target is t.v.s}, we have preferred to keep the two results separated: their proofs differ, and the one for the first result is more quantitative in nature (due to the existence of a family of seminorms  $\{p_\lambda;\,\lambda\in\Lambda\}$ generating the topology), which may pave the way for more precise quantitative results and estimates: we leave this kind of questions for future works. Besides, also in this case, we show that the resulting approximating neural network can be replaced with a ``finite-dimensional'' version thereof (namely, specified by a finite number of parameters) if in addition the space $E$ is assumed to admit a pre-Schauder basis: see Proposition \ref{prop: finite-dim approx for E t.v.s} and Corollary \ref{corr: finite-dim approx for E t.v.s} for details.

\subsubsection{The target space is a locally convex space}

Let $(\X,\tau)$ be a Hausdorff topological space and $(E,\tau_E)$ a locally convex space (not necessarily Hausdorff), with a family of seminorms $\{p_\lambda;\,\lambda\in\Lambda\}$ generating its topology. Let us consider $C(\X,E)$
    the space of continuous functions from $\X$ to $E$ endowed with the topology of compact convergence. This means that a 0-neighborhood base for this locally convex topology is given by 
    \[
    \{
    h\in C(\X,E);\,\, \sup_{x\in \mathcal{K}} p_{\lambda_1}(h(x))<\varepsilon_1,\dots,
    \sup_{x\in \mathcal{K}} p_{\lambda_R}(h(x))<\varepsilon_R
    \}
    \]
    for $R\in\N,\lambda_1,\dots,\lambda_R\in\Lambda,\varepsilon_1,\dots,\varepsilon_R>0$ and $\mathcal K$ running among the compact subsets of $\X$. Alternatively, $p_{\mathcal{K},\lambda}(h):=\sup_{x\in\mathcal K}p_\lambda(h(x)),\,h\in C(\X,E)$, with $\mathcal{K}$ compact and $\lambda\in\Lambda$ are a set of continuous seminorms generating this topology. 

    We are ready to state and prove
   
\begin{proposition}\label{prop: the target is l.c.s}
    Let $(\X,\tau)$ be quasi-Polish, $(f_i)_{i\in\N}$ be a separating sequence, and as usual set $F=(f_1,f_2,\dots)$. Let $(E,\tau_E)$ be a locally convex space, with a family of seminorms $\{p_\lambda;\,\lambda\in\Lambda\}$ that generates $\tau_E$. Assume that the activation function $\sigma:V\to V$ is continuous, satisfying the separating property \eqref{eq: abstract condition on sigma} and such that $\sigma(V)$ is bounded. Let $C(\X,E)$ be
    the space of continuous functions from $\X$ to $E$ endowed with the topology of compact convergence. Then $\mathfrak{N}_{F,\sigma}(\X,E)$ is dense in $C(\X,E)$.
\end{proposition}

\begin{proof}
    Let $g:\X\to E$ be continuous. We fix once for all a compact subset $\mathcal{K}\subset\X$ and a neighborhood of 0
    \[
    U(p_{\lambda_1},\dots, p_{\lambda_R};\varepsilon_1,\dots ,\varepsilon_R) = \{
    z\in E; \,\, p_{\lambda_1}(z)<\varepsilon_1,\dots ,p_{\lambda_R}(z)<\varepsilon_R
    \}
    \]
    in $E$. From \cite{shuchat,shuchat2}, we know that $C(\mathcal{K})\otimes E$
    is dense in  $C(\mathcal{K},E)$, where $\mathcal{K}$ is clearly endowed with the subspace topology inherited from $\X$. Thus, we may find $\Bar{g}\in C(\mathcal K)\otimes E$ such that 
    \[
    g(x) - \Bar{g}(x) \in U(p_{\lambda_1},\dots, p_{\lambda_R};\varepsilon_1/2,\dots ,\varepsilon_R/2),\quad x\in\mathcal{K}.
    \]
    Such $\Bar{g}$ can be written as $\Bar{g}(x)=\sum_{i=1}^M\Bar{g}^{(i)}(x)v^{(i)}$, with suitable $M\in\N,\Bar{g}^{(i)}\in C(\mathcal{K})$ and $v^{(i)}\in E$.
    Since evidently $\mathcal{K}$ is also a quasi Polish space with the same separating sequence $(f_i)_{i=1}^\infty$ of $\X$, then by Theorem \ref{thm: first} it is possible to find $M$ neural networks $\mathcal{N}^{(1)},\dots,\mathcal{N}^{(M)}\in\mathfrak{N}_{F,\sigma}(\mathcal{K})$,
    \[
    \mathcal{N}^{(i)}(x)=\sum_{h=1}^{H^{(i)}}\langle \phi^{(i)}_h,\sigma (A_h^{(i)}F(x) + b^{(i)}_h)\rangle, \quad x\in \mathcal{K},\, i=1,\dots , M
    \]
    with $\phi^{(i)}_h\in V',\,A_h^{(i)}\in\mathcal{L}( V)$ and $b^{(i)}_h\in V$, such that
    \[
    \abs{\mathcal{N}^{(i)}(x) - \Bar{g}^{(i)}(x)} < \delta_i,\quad x\in\mathcal{K},\,i=1,\dots ,M,
    \]
    where we have set 
    \[
    \delta_i:=\frac{1}{2M}\frac{\Bar{\varepsilon} }{ c_i + 1},\quad i=1,\dots,M,
    \]
    with $\Bar{\varepsilon}:=\min_{j=1,\dots,R}\varepsilon_j,\,c_i:= \max_{j=1,\dots,R} p_{\lambda_j}(v^{(i)}) $.
    
    We define $\mathcal{N}(x)=\sum_{i=1}^M\mathcal{N}^{(i)}(x)v^{(i)},\,x\in\mathcal{K}$; so $\mathcal{N}\in\mathfrak{N}_{F,\sigma}(\mathcal{K},E)$. For each $j=1,\dots,R$ and $x\in\mathcal{K}$, it holds
    \[
    \begin{split}
       p_{\lambda_j}(\mathcal{N}(x)-\Bar{g}(x)) &= p_{\lambda_j}\left(
       \sum_{i=1}^M\mathcal{N}^{(i)}(x)v^{(i)} - \sum_{i=1}^M\Bar{g}^{(i)}(x)v^{(i)}
       \right)\\
       &\le \sum_{i=1}^M\abs{\mathcal{N}^{(i)}(x) - \Bar{g}^{(i)}(x)} p_{\lambda_j}(v^{(i)})\\
       &<\sum_{i=1}^M  \frac{1}{2M}\frac{ \Bar{\varepsilon} }{ c_i + 1}p_{\lambda_j}(v^{(i)}) \\
       &<\frac{1}{2M} \bar{\varepsilon}\sum_{i=1}^M 1\\ 
       &= \frac{1}{2} \min_{j=1,\dots,R}\varepsilon_j.
       \end{split}
    \]
    Therefore, for each $j=1,\dots,R$ and $x\in\mathcal{K}$, it holds
    \[
    \begin{split}
        p_{\lambda_j}(\mathcal{N}(x) - g(x)) &\le p_{\lambda_j}(\mathcal{N}(x) - \Bar{g}(x)) + p_{\lambda_j}(\Bar{g}(x) - g(x)) \\
        &< \frac{1}{2} \min_{j=1,\dots,R}\varepsilon_j + \frac{\varepsilon_j}{2}\\
        &\le \varepsilon_j,
    \end{split}
    \]
    namely, $\mathcal{N}(x)-g(x)\in U(p_{\lambda_1},\dots, p_{\lambda_R};\varepsilon_1,\dots ,\varepsilon_R) $ for each $x\in\mathcal{K}$. Finally, by construction, the neural networks $\mathcal{N}^{(i)}(x)=\sum_{h=1}^{H^{(i)}}\langle \phi^{(i)}_h,\sigma (A_h^{(i)}F(x) + b^{(i)}_h)\rangle$ can be naturally extended to elements of $\mathfrak{N}_{F,\sigma}(\X)$, because the map $F$ is already defined on the whole $\X$. Thus, we conclude that actually $\mathcal{N}\in\mathfrak{N}_{F,\sigma}(\X,E)$ and that $\mathfrak{N}_{F,\sigma}(\X,E)$ is dense in $C(\X,E)$.

\end{proof}

\subsubsection{The target space is a topological vector space}

Let us now generalize the previous result to the topological vector space category. More precisely, consider again $(\X,\tau)$ a Hausdorff topological space and let $(E,\tau_E)$ be a topological vector space (not necessarily Hausdorff). We want to recall here how to endow the space $C(\X,E)$ with the topology of uniform convergence on compact subsets which renders it a topological vector space. On the vector space $E^\X$ (i.e. the set of all maps from $\X$ into $E$), we place the topology of uniform convergence on compact subsets (also known as the compact-open topology), for which a 0-neighborhood base is provided by the subsets
\[
\{
u \in E^\X;\, u(\mathcal K)\subset\mathcal{V}\},
\]
where $\mathcal{K}$ runs among the compact subsets of $\X$ and $\mathcal{V}$ among a 0-neighborhood base in $E$. Clearly, this topology does not depend on the particular choice of the 0-neighborhood base in $E$, and it is translation-invariant. Besides, instead of considering the familiy of all compact subsets of $\X$, we may consider a sub-family $\mathfrak C$ thereof with the property that for any compact $\mathcal{K}$ there exists another compact $\mathcal{K}'\in\mathfrak C$ such that $\mathcal K \subset \mathcal K'$, and the resulting topology on $E^\X$ would not change.

We endow the vector subspace $C(\X,E)$ of $E^\X$ with this topology. Since for any $u\in C(\X,E)$ and $\mathcal{K}$ compact, $u(\mathcal{K})$ is compact, and hence totally bounded (see e.g. \cite{Mangatiana_CompactnessPrinciplesTVS}) and thus bounded (\cite[page 25]{schaefer1999topological}), we conclude from 
\cite[page 79]{schaefer1999topological} that $C(\X,E)$ with the topology of uniform convergence on compacts is a topological vector space.

As a preparation to prove Proposition \ref{prop: the target is t.v.s}, we need this easy lemma: 
\begin{lemma}\label{lemma: change of variable}
    Let $(T,\tau)$ and $(T',\tau')$ be two compact homeomorphic Hausdorff spaces. Let $\varphi:T'\to T$ be a homeomorphism. Let $(E,\tau_E)$ be a topological vector space and consider the topological vector spaces $C(T,E)$ and $C(T',E)$ endowed with the topology of uniform convergence. Then $C(T,E)$ and $C(T',E)$ are isomorphic as topological vector spaces. The same holds for the subspaces $C(T)\otimes E$ and $C(T')\otimes E$. In particular, $C(T)\otimes E$ is dense in $C(T,E)$ if and only if $C(T')\otimes E$ is dense in $C(T',E)$.

\end{lemma}

\begin{proof}
    We consider the operator 
    \[
    \Phi:C(T,E)\to C(T',E),\quad C(T,E)\ni u\stackrel{\Phi}{\mapsto} u\circ \varphi \in C(T',E) 
    \]
    which is clearly well-defined and linear. Let us prove that it is continuous: being $T$ already compact, a 0-neighborhood base for the topology of uniform convergence on compacts can be provided in this case by 
\[
\{
u \in E^T;\, u(T)\subset\mathcal{U}\},
\]
where $\mathcal{U}$ runs among a 0-neighborhood base in $E$. A similar argument holds for $E^{T'}$ as well. So, given a net $<u_\gamma>_\gamma\subset C(T,E)$ converging uniformly to 0, and a 0-neighborhood $\{
v \in E^{T'};\, v(T')\subset\mathcal{U}\}$
for some open 0-neighborhood $\mathcal{U}$ in $E$, we can then find an index $\gamma_0$ such that for $\gamma \succcurlyeq \gamma_0$ it holds $u_\gamma(t)\in \mathcal{U},\,t\in T$, and thus $u_\gamma(\varphi(t'))\in\mathcal{U},\,t'\in T'$. Namely, $u_\gamma\circ\varphi\in \{
v \in E^{T'};\, v(T')\subset\mathcal{U}\} $
if  $\gamma \succcurlyeq \gamma_0$ , i.e. $\Phi u_\gamma\to 0$ uniformly. Besides, injectivity and surjectivity are evident. Let $\Phi^{-1}:C(T',E)\to C(T,E),\,v\stackrel{\Phi^{-1}}{\longmapsto}v\circ \varphi^{-1}$ be the inverse. Continuity is proved exactly as above, and hence,
    \[
    C(T,E)\cong C(T',E).
    \]
    Similarly, by restricting the operator $\Phi$ to the vector subspace $C(T)\otimes E$, it is clear that $\Phi(C(T)\otimes E)\subset C(T')\otimes E$
    and continuity and injectivity are evident. Surjectivity also holds, because given $\sum_ia_i(t')\otimes x_i\in C(T')\otimes E$, we have that its pre-image via $\Phi$ is $$\sum_ia_i(\varphi^{-1}(t))\otimes x_i\in C(T)\otimes E.$$ Continuity of $\Phi^{-1}: C(T')\otimes E \to C(T)\otimes E$ is also easily proven, and in conclusion we obtain
    \[
    C(T)\otimes E \cong C(T')\otimes E.
    \]
    The last claim is now clear.
    
\end{proof}

We are ready to prove

\begin{proposition}\label{prop: the target is t.v.s}

Let $(\X,\tau)$ be quasi-Polish, $(f_i)_{i\in\N}$ be a separating sequence, and $F$ be the associated injection. Assume that the activation function $\sigma:V\to V$ is continuous, satisfies the separating property \eqref{eq: abstract condition on sigma} and such that $\sigma(V)$ is bounded. Let $(E,\tau_E)$ be a topological vector space, and let $C(\X,E)$ be the space of continuous functions from $\X$ to $E$ endowed with the topology of uniform convergence on compacts. Then $\mathfrak{N}_{F,\sigma}(\X,E)$ is dense in $C(\X,E)$.
\end{proposition}
\begin{proof}
For an arbitrary compact set $\mathcal K \subset\X$, we know that $F(\mathcal K)\subset\mathcal{Q}\subset V$ is homeomorphic to $\mathcal{K}$ via $F$, and so in particular $F(\mathcal K)$ is compact. Since $V$ is a Hilbert space, it follows that the identity map of $V$ restricted to $F(\mathcal K)$ can be uniformly approximated by continuous maps with finite-dimensional range (e.g. by orthogonal projections associated to a orthonormal basis of $ V$). \cite[Corollary 1, page 99]{shuchat} then applies, and so it holds
\[
C(F(\mathcal{K}),E) = \overline{C(F(\mathcal K))\otimes E},
\]
where clearly the closure is taken with respect to the topology of uniform convergence. By means of Lemma \ref{lemma: change of variable}, we then have
\[
C(\mathcal K,E) = \overline{C(\mathcal K)\otimes E}.
\]
By \cite[1.2 page 14]{schaefer1999topological}, we may find a 0-neighborhood base $\mathcal B$ for $E$ with the following
\begin{enumerate}
    \item for any $\mathcal V\in \mathcal B$ there exists $\mathcal U\in \mathcal{B}$ such that $\mathcal U + \mathcal U \subset \mathcal V$,
    \item any $\mathcal V\in \mathcal B$ is radial and circled,
    \item there exists $\lambda\in\R,\,0<\abs{\lambda}<1$ such that $\mathcal V\in \mathcal B$ implies $\lambda\mathcal V\in \mathcal B$.
\end{enumerate}
We are going to work with this base. Let $g\in C(\X,E)$ and $\mathcal{V}\in\mathcal{B}$: then it is possible to find $\mathcal U\in\mathcal{B}$ such that $\mathcal U + \mathcal U \subset \mathcal V$ (in particular, $\mathcal U = \mathcal U + 0 \subset \mathcal V$). Thus, we may choose $\Bar{g}\in C(\mathcal{K})\otimes E$ such that 
\[
g(x)-\Bar{g}(x)\in\mathcal{U},\quad x\in\mathcal{K},
\]
i.e. $(g-\Bar{g})(\mathcal{K})\subset\mathcal{U}$. Such $\Bar{g}$ can be written as $\Bar{g}(x)=\sum_{i=1}^M\Bar{g}^{(i)}(x)v^{(i)}$, with suitable $M\in\N,\Bar{g}^{(i)}\in C(\mathcal{K})$ and $v^{(i)}\in E$. We can now find $\mathcal W\in\mathcal B$ such that
\[
\underbrace{\mathcal W + \dots + \mathcal W}_{M\text{ times}}\subset\mathcal{U}.
\]
This fact easily follows from (1) above. Since the scalar multiplication is continuous, there exist $\delta_i>0,\,i=1,\dots,M$, such that $\delta_iv^{(i)}\in\mathcal W,\,i=1,\dots,M$. 

Our goal now is to approximate each $\Bar{g}^{(i)}$ with suitable neural architectures. Since also $\mathcal{K}$, endowed with the relative topology, is a quasi-Polish space with the same separating sequence $(f_i)_{i\in\N}$ of $\X$, by Theorem \ref{thm: first} it is possible to find $M$ neural networks in $\mathfrak{N}_{F,\sigma}(\mathcal{K})$ like this
\[
    \mathcal{N}^{(i)}(x)=\sum_{h=1}^{H^{(i)}}\langle \phi^{(i)}_h,\sigma (A_h^{(i)}F(x) + b^{(i)}_h)\rangle, \quad x\in \mathcal{K},\, i=1,\dots , M
\]
with $\phi^{(i)}_h\in V',\,A_h^{(i)}\in\mathcal{L}( V)$ and $b^{(i)}_h\in V$, such that
\[
    \abs{\mathcal{N}^{(i)}(x) - \Bar{g}^{(i)}(x)} < \delta_i,\quad x\in\mathcal{K},\,i=1,\dots ,M.
\]
Define $\mathcal{N}(x)=\sum_{i=1}^M\mathcal{N}^{(i)}(x)v^{(i)},\,x\in\mathcal{K}$; so $\mathcal{N}\in\mathfrak{N}_{F,\sigma}(\mathcal{K},E)$. Write the difference between $\Bar{g}(x)$ and $\mathcal{N}(x)$ as
\[
\Bar{g}(x) - \mathcal{N}(x) = \sum_{i=1}^M\left(\Bar{g}^{(i)}(x) -\mathcal{N}^{(i)}(x) \right)v^{(i)},\quad x\in\mathcal{K},
\]
and observe that, since
\[
\left(\Bar{g}^{(i)}(x) -\mathcal{N}^{(i)}(x) \right)v^{(i)} = 
\frac{\Bar{g}^{(i)}(x) -\mathcal{N}^{(i)}(x)}{\delta_i}\cdot \underbrace{\delta_i v^{(i)}}_{\in\mathcal{W}}
\]
and $\abs{\frac{\Bar{g}^{(i)}(x) -\mathcal{N}^{(i)}(x)}{\delta_i}}\le 1$, then $\left(\Bar{g}^{(i)}(x) -\mathcal{N}^{(i)}(x) \right)v^{(i)}\in\mathcal W$ for $x\in\mathcal{K},\,i=1,\dots,M$, because $\mathcal W$ was circled. Therefore, $\Bar{g}(x) - \mathcal{N}(x)\in\mathcal W+\dots + \mathcal W\subset\mathcal{U}$ for $x\in\mathcal K$, and we conclude that
\[
g(x)-\mathcal{N}(x) = g(x) -\Bar{g}(x) + \Bar{g}(x)-\mathcal{N}(x)\in\mathcal{U}+\mathcal{U}\subset\mathcal{V},\quad x\in\mathcal{K}.
\]
Finally, by construction, the neural networks $\mathcal{N}^{(i)}(x)=\sum_{h=1}^{H^{(i)}}\langle \phi^{(i)}_h,\sigma (A_h^{(i)}F(x) + b^{(i)}_h)\rangle$ can be naturally extended to elements of $\mathfrak{N}_{F,\sigma}(\X)$, because the map $F$ is already defined on the whole $\X$. Thus, we conclude that actually $\mathcal{N}\in \mathcal{N}\in\mathfrak{N}_{F,\sigma}(\X,E)$ and that $\mathfrak{N}_{F,\sigma}(\X,E)$ is dense in $C(\X,E)$, because for any $g\in C(\X,E)$, any compact $\mathcal{K}\subset\X$ and any $\mathcal{V}\in\mathcal B$, we have found $\mathcal N\in\mathfrak{N}_{F,\sigma}(\X,E)$ such that
\[
(g-\mathcal{N})(\mathcal{K})\subset\mathcal{V}.
\]

\end{proof}

Similarly to Theorem \ref{thm: second}, also in the present case we can replace the resulting approximating neural network with a ``finite-dimensional'' version thereof. As before, we fix an orthonormal basis $(e_k)_{k\in\N}$ for $V$, and $\Pi_N$ denotes the orthogonal projection onto $\Span\{e_1,\dots ,e_N\}$. We are ready to prove

\begin{proposition}\label{prop: finite-dim approx for E t.v.s}
Let $(\X,\tau)$ be quasi-Polish, $(f_i)_{i\in\N}$ be a separating sequence, and set $F=(f_1,f_2,\dots)$.
Assume that the activation function $\sigma:V\to V$ is Lipschitz continuous, satisfying the separating property \eqref{eq: abstract condition on sigma} and such that $\sigma(V)$ is bounded. Let $(E,\tau_E)$ be a topological vector space, $g:\X\to E$ a continuous function, $\mathcal{K}\subset\X$ compact and $\mathcal{U}_0$ a 0-neighborhood in $E$. Let $\mathcal{N}=\sum_{i=1}^M\mathcal{N}^{(i)}v^{(i)}\in \mathfrak N_{F,\sigma}(\X,E)$ be such that
\[
g(x)-\mathcal{N}(x)\in\mathcal{U}_0,\quad x\in\mathcal{K}.
\]
Let $\mathcal{U}_1$ be a 0-neighborhood in E. Then there exists $\bar{N}\in\N$ such that for $N\ge\Bar{N}$ we have
\[
g(x) - \sum_{i=1}^M \Lambda_N(\mathcal{N}^{(i)})(x)v^{(i)}\in \mathcal{U}_0+\mathcal{U}_1 ,\quad x\in\mathcal{K}.
\]
In other words, $\bigcup_{N\in\N}\mathfrak{N}_{F,\sigma,N}(\X,E)$ is dense in $C(\X,E)$.

\end{proposition}

\begin{proof}
Without loss of generality, we can assume that $\mathcal{U}_1\in\mathcal{B}$, where $\mathcal{B}$ is the 0-neighborhood base provided by \cite[1.2 page 14]{schaefer1999topological}: compare the proof of Proposition \ref{prop: the target is t.v.s}. Then we can find $\mathcal W\in\mathcal B$ such that
\[
\underbrace{\mathcal W + \dots + \mathcal W}_{M\text{ times}}\subset\mathcal{U}_1.
\]
and $\delta_i>0,\,i=1,\dots,M$, such that $\delta_iv^{(i)}\in\mathcal W,\,i=1,\dots,M$, because the scalar multiplication is continuous. By Theorem \ref{thm: second}, equation \eqref{eq: second (2)}, there exist $N_i\in\N,i=1,\dots,M$, such that for $N\ge N_i$ it holds good
\[
\sup_{x\in\mathcal{K}}\abs{ \mathcal{N}^{(i)}(x)-\Lambda_N(\mathcal{N}^{(i)})(x)
}<\delta_i.
\]
Set 
\[
\Bar{N}:=\max\{N_1,\dots,N_M\}.
\]
We may write
\[
\mathcal{N}(x) - \sum_{i=1}^M\Lambda_N(\mathcal{N}^{(i)})(x) v^{(i)} = \sum_{i=1}^M
\frac{ \mathcal{N}^{(i)}(x) -\Lambda_N(\mathcal{N}^{(i)})(x)}{\delta_i}\cdot \underbrace{\delta_i v^{(i)}}_{\in\mathcal{W}}
\]
Then for $N\ge\Bar{N}$ and $x\in\mathcal{K}$, since $\abs{\frac{\mathcal{N}^{(i)}(x) -\Lambda_N(\mathcal{N}^{(i)})(x)}{\delta_i}}\le 1$ and $\mathcal{W}$ is circled, we deduce that $\left(\mathcal N^{(i)}(x) -\Lambda_N(\mathcal{N}^{(i)})(x) \right)v^{(i)}\in\mathcal W$, for $i=1,\dots,M$, and therefore  
\[
\mathcal{N}(x) - \sum_{i=1}^M\Lambda_N(\mathcal{N}^{(i)})(x) v^{(i)}\in\mathcal{U}_1.
\]
We conclude that for $N\ge\Bar{N}$
\[
g(x) - \sum_{i=1}^M\Lambda_N(\mathcal{N}^{(i)})(x) v^{(i)}\in\mathcal{U}_0+\mathcal{U}_1,\quad x\in\mathcal{K}.
\]

\end{proof}

    Also here we see that the scalar neural networks $\Lambda_N(\mathcal{N}^{(i)}),\,i=1,\dots,M$ are ``finite-dimensional'', and therefore readily implementable at the numerical level: refer to the discussion immediately after Theorem \ref{thm: second}. However, in the vectorial case, there is a second layer of approximation, given by the vectors $v^{(i)}$ which apriori can span the whole space $E$: this in general is a drawback, because the space $E$ can be ``very large''. Therefore, if we want to obtain neural architectures that are implementable in practice (namely, fully specified by a finite number of parameters), we are required to impose extra assumptions on the ``size'' of $E$. We accomplish this, by requiring that the topological vector space $(E,\tau_E)$ admits a pre-Schauder basis (and so in particular it is separable). This is a simple consequence of the following lemma:

\begin{lemma}\label{lemma: E admits a pre-Schauder}
Let $(E,\tau_E)$ be a topological vector space carrying a pre-Schauder basis $(s_k)_{k\in\N}\subset E$, and let $\Pi^E_N,\,N\in\N$ be the canonical projections associated to it. Let $(T,\tau)$ be a compact Hausdorff space and $u\in C(T)\otimes E$
\[
u(t)=\sum_{i=1}^Mu_i(t)w_i,\quad t\in T
\]
for some $M\in\N,0\neq u_i\in C(T)$ and $w_i\in E$. Then for any 0-neighborhood $\mathcal{V}$ in $E$, there exists $N_0\in\N$ such that for all $N\ge N_0$ it holds
\[
u(t)-\sum_{i=1}^Mu_i(t)\Pi^E_Nw_i
\in \mathcal{V},\quad t\in T.
\]
   
\end{lemma}
\begin{proof}
    Clearly, it is enough to work with the 0-neighborhood base $\mathcal{B}$ used in Proposition \ref{prop: the target is t.v.s}. So, let us assume that $\mathcal V\in \mathcal B$: we can find $\mathcal{W}\in\mathcal{B}$ such that
    \[
\underbrace{\mathcal W + \dots + \mathcal W}_{M\text{ times}}\subset\mathcal{V}.
\]
Set $\norm{u_i}_\infty:=\sup_{t\in T}\abs{u_i(t)}\in (0,\infty)$. Since $(s_k)_{k\in\N}$ is pre-Schauder basis, for each $w_i$ there exists $N_i\in\N$ such that, whenever $N\ge N_i$, 
\[
w_i-\sum_{k=1}^N\beta_k^E(w_i)s_k = w_i - \Pi^E_N(w_i) \in \norm{u_i}_\infty^{-1}\mathcal W.
\]
In this way, since $\abs{\frac{u_i(t)}{\norm{u_i}_\infty}}\le 1,\,t\in T$ and $\mathcal W$ is circled, we have
\[
u_i(t)(w_i-\Pi^E_Nw_i) = \frac{u_i(t)}{\norm{u_i}_\infty} \underbrace{ \norm{u_i}_\infty(w_i- \Pi^E_Nw_i)}_{\in\mathcal W}\in\mathcal{W},\quad t\in T,N\ge N_i.
\]
We set $N_0:=\max_{i=1,\dots, M}N_i$ and we conclude that for all $N\ge N_0$
\[
u(t)-\sum_{i=1}^Mu_i(t)\Pi^E_Nw_i=\sum_{i=1}^Mu_i(t)(w_i-\Pi^E_Nw_i)\in\mathcal{V},\quad t\in T.
\]

\end{proof}

Therefore, if in Proposition \ref{prop: finite-dim approx for E t.v.s} we now assume additionally that $E$ admits a pre-Schauder basis $(s_k)_{k\in\N}$, then we can achieve the goal of obtaining neural network architectures specified by a finite number of parameters.

\begin{corollary}\label{corr: finite-dim approx for E t.v.s}
Assume the same setting of Proposition \ref{prop: finite-dim approx for E t.v.s} and that $E$ admits a pre-Schauder basis $(s_k)_{k\in\N}$.
Let $g:\X\to E$ be a continuous function, $\mathcal{K}\subset\X$ compact and $\mathcal{U}_0$ a 0-neighborhood in $E$. Let $\mathcal{N}=\sum_{i=1}^M\mathcal{N}^{(i)}v^{(i)}\in \mathfrak N_{F,\sigma}(\X,E)$ be such that
\[
g(x)-\mathcal{N}(x)\in\mathcal{U}_0,\quad x\in\mathcal{K}.
\]
Let $\mathcal{U}_1$ be a 0-neighborhood in E. Then there exists $\bar{N}\in\N$ such that for $N\ge\Bar{N}$ we have
\[
g(x) - \sum_{i=1}^M\Lambda_N(\mathcal{N}^{(i)})(x)\Pi^E_N v^{(i)}\in \mathcal{U}_0+\mathcal{U}_1 ,\quad x\in\mathcal{K}.
\]

\end{corollary}
\begin{proof}
    We choose $\mathcal{W}_0\in\mathcal B$ such that $\mathcal{W}_0+\mathcal{W}_0\subset\mathcal U_1$ and $\mathcal{W}\in\mathcal{B}$ such that 
    \[
\underbrace{\mathcal W + \dots + \mathcal W}_{M\text{ times}}\subset\mathcal{W}_0.
\]
Set $a_i:=\sup_{x\in\mathcal{K}}\abs{\mathcal{N}^{(i)}(x)}\in (0,\infty),\,i=1,\dots, M$, and find $N_0\in\N$ such that for all $N\ge N_0$ we have
\[
v^{(i)}-\Pi_N^Ev^{(i)}\in a_i^{-1}\mathcal{W},\quad i=1,\dots,M.
\]
In this way, arguing as above, we obtain
\begin{equation}\label{eq: corr: finite-dim approx for E t.v.s (1)}
    \sum_{i=1}^M\mathcal{N}^{(i)}(x)[v^{(i)}-\Pi_N^Ev^{(i)}]\in\mathcal{W}_0.
\end{equation}
We choose $\mathcal{V}\in\mathcal B$ such that $\mathcal V + \mathcal{V}\subset\mathcal{W}$, and $\delta_i>0$ such that $\delta_i v^{(i)}\in\mathcal{V},i=1,\dots,M$. We consequently find $N_1\in\N$ such that $N\ge N_1$ implies
\[
v^{(i)}-\Pi_N^Ev^{(i)}\in\delta_i^{-1}\mathcal V,\quad i=1,\dots,M.
\]
In this way, for $N\ge N_1$
\[
\delta_i\Pi_N^Ev^{(i)} = \delta_i(\Pi_N^Ev^{(i)}-v^{(i)}) + \delta_iv^{(i)}\in\mathcal{V} + \mathcal{V}\subset\mathcal{W},\quad i=1,\dots,M.
\]
As in the proof of Proposition \ref{prop: finite-dim approx for E t.v.s}, we can now find $N_2\in\N$ such that, if $N\ge N_2$,
\[
\sup_{x\in\mathcal K}\abs{\mathcal{N}^{(i)}(x) -\Lambda_N(\mathcal{N}^{(i)})(x)} < \delta_i,\quad i=1,\dots, M,
\]
and hence, for $N\ge \Bar{N}:=\max\{N_0,N_1,N_2\}$
\begin{equation}\label{eq: corr: finite-dim approx for E t.v.s (2)}
    \sum_{i=1}^M(\mathcal{N}^{(i)}(x) -\Lambda_N(\mathcal{N}^{(i)})(x) )\Pi_N^Ev^{(i)}\in \underbrace{\mathcal W + \dots + \mathcal W}_{M\text{ times}}\subset\mathcal{W}_0.
\end{equation}
Using \eqref{eq: corr: finite-dim approx for E t.v.s (1)} and \eqref{eq: corr: finite-dim approx for E t.v.s (2)}, we finally obtain ($N\ge\Bar{N}$)
\[
\begin{split}
    &g(x) - \sum_{i=1}^M\Lambda_N(\mathcal{N}^{(i)})(x)\Pi_N^Ev^{(i)}\\
    &\quad= g(x)-\mathcal{N}(x) + \sum_{i=1}^M\mathcal{N}^{(i)}(x)[v^{(i)}-\Pi_N^Ev^{(i)}] + \sum_{i=1}^M[\mathcal{N}^{(i)}(x) - \Lambda_N(\mathcal{N}^{(i)})(x) ]\Pi_N^Ev^{(i)}\\
    &\quad\quad\in\mathcal{U}_0+\mathcal{W}_0+\mathcal{W}_0\\
    &\quad\quad\subset\mathcal{U}_0 +\mathcal{U}_1,\quad x\in\mathcal{K}.
\end{split}
\]
    
\end{proof}

\subsection{Universal approximation results for targets that are quasi-Polish}\label{subsec: Universal approximation results for targets that are quasi-Polish}

As anticipated above, we can also have as a output space a second quasi Polish space, namely something that in general does not possess a linear structure. The price we have to pay is however that the resulting neural architectures will be only Borel measurable in general, even though with finite range. The main reason for this is based on the strategy we are going to apply: given two quasi-Polish spaces $\X$ and $\Y$, and a continuous function $g:\X\to \Y$, we will map the range of $g$ into $V$ via the injection map of $\Y$ (call it $H$), and we will apply our previous result (Proposition \ref{prop: the target is l.c.s}) to obtain an approximating neural network $\mathcal N$ with range in $V$. In order to eventually obtain an architecture with range in the original space $\Y$, we will need to pull back $\mathcal N$ to $\Y$ via $H^{-1}$. However, the range of $\mathcal{N}$ in general falls outside $H(\Y)$, and hence we are required  first to make a projection onto $H(\Y)$ is some way. Since in general this subset of $V$ will not be convex, we cannot use the classical theory of projections in Hilbert spaces, and we will have to resort to a variant of the metric projection in the spirit of Voronoi cells. This projection will always be of finite range but in general will fail to be continuous: refer also to Subsection \ref{subsec: metric projection} 

We would like to remark here that the result we are going to present now overlaps only partially with the previous ones where the target space was assumed to be a topological vector space. Indeed, it is trivial to see that q.P $\not\subset$ t.v.s, and, on the other hand, it also true that t.v.s $\not\subset$ q.P, because quasi-Polish spaces are functional Hausdorff and hence Hausdorff.

First of all, we want to recall here and prove this result concerning projections in metric spaces.
\begin{lemma}\label{lemma: projection in metric spaces}
    Let $(Z,d_Z)$ be a metric space, $m\in\N$ and $a_1,\dots,a_m$ distinct elements of $Z$. For any $z\in Z$, define
    \[
    j(z)=\min\left\{j\in \{1,\dots,m\};\;d_Z(z,a_j) = d_Z(z;\{a_1,\dots,a_m\}) \right\}
    \]
    and the map
    \[
    P^Z_{a_1,\dots,a_m}: Z\to Z,\quad z\mapsto a_{j(z)}.
    \]
    Then $P^Z_{a_1,\dots ,a_m}$ is $\mathcal{B}(Z)/\mathcal{B}(Z)$-measurable.
\end{lemma}
\begin{proof}
    We briefly recall its proof here. Set $D_k:Z\to \R,\,z\mapsto D_k(z):=d_Z(z,a_k)$ for $k=1,\dots ,m$: clearly, they are continuous. Besides, it is easy to see that for $k=2,\dots,m$, it holds
\[
 (P^Z_{a_1,\dots,a_m})^{-1}(a_k)= \left\{z\in Z;\, D_k(z) \le \min_{u=1,\dots, m} D_u(z)
\right\} \cap \left\{z\in Z;\, D_k(z) < \min_{u=1,\dots, k-1}D_u(z)
\right\}
\]
which results in an intersection of a closed and an open set. For $a_1$ it holds instead
\[
(P^Z_{a_1,\dots,a_m})^{-1}(a_1) = \left\{z\in Z;\, D_1(z) \le \min_{u=1,\dots, m} D_u(z)
\right\} 
\]
which is closed. This proves the lemma.

\end{proof}

In the following, if $Z=V$, we will simply write $P_{a_1,\dots,a_m}$
rather than $P^V_{a_1,\dots,a_m}$. Consider now two quasi-Polish spaces $(\X,\tau_\X,(f_n)_{n=1}^\infty)$ and $(\Y,\tau_\Y,(h_n)_{n=1}^\infty)$ and let 
\[
F=(f_1,f_2,\dots):\X\to \mathcal{Q}\subset V
\]
\[
H=(h_1,h_2,\dots):\Y\to \mathcal{Q}\subset V
\]
be the induced maps. We define the set of \textbf{quasi-Polish Borel neural networks} from $\X$ to $\Y$ as
\[
\begin{split}
    &\mathfrak{BN}_{F,H,\sigma}(\X,\Y)\\
    &\quad:=\left\{H^{-1}\circ P_{a_1,\dots,a_R}\circ \mathcal{N}:\X\to\Y;\;\mathcal{N}\in\mathfrak{N}_{F,\sigma}(\X, V),R\in\N,a_1\,\dots,a_R\in H(\Y)\right\}.
\end{split}
\]

We observe that the map $H^{-1}\circ P_{a_1,\dots,a_R}\circ \mathcal{N}$ is indeed $\mathcal{B}(\X)/\mathcal{B}(\Y)$-measurable because $\mathcal{N}$ is continuous, $P_{a_1,\dots,a_R}$ is Borel measurable and $H$ is a homeomorphism from the compact set $\{H^{-1}(a_1),\dots , H^{-1}(a_R) \}\subset\Y$ and $\{a_1,\dots , a_R \}\subset V$. Define on $\Y$ the metric
\[
    d_H:\Y\times\Y\to[0,\infty),\quad (y_1,y_2)\mapsto d_H(y_1,y_2)=\norm{H(y_1)-H(y_2)}_V,
    \]
whose induced topology restricted to compact subsets of $\Y$ coincide with the original topology $\tau_\Y$    
(see Lemma \ref{lemma: metrizability of compacts}). We have:
\begin{proposition}\label{prop: target q.P}
Assume that the activation function $\sigma:V\to V$ is continuous, satisfying the separating property \eqref{eq: abstract condition on sigma} and such that $\sigma(V)$ is bounded. Then given an arbitrary $g\in C(\X,\Y)$, a compact subset $\mathcal{K}\subset\X$ and an error $\varepsilon>0$, there exists a neural network $\mathcal{M}\in\mathfrak{BN}_{F,H,\sigma}(\X,\Y) $, $\mathcal{M}=H^{-1}\circ P_{a_1,\dots,a_R}\circ \mathcal{N}$ with suitable $\mathcal{N}\in\mathfrak{N}_{F,\sigma}(\X, V)$, $R\in\N$ and $a_1,\dots,a_R\in H\circ g(\mathcal{K})\subset\mathcal{Q}\subset V$ such that
   \[
   d_H\left(g(x),\mathcal{M}
   (x)\right)<\varepsilon,\quad x\in \mathcal{K}.
   \]
\end{proposition}
\begin{proof}
    We consider the map $H\circ g:\X\to \mathcal{Q}\subset V$ which is continuous by composition; in particular, $H\circ g(\mathcal{K})$ is compact. By Proposition \ref{prop: the target is l.c.s} (say), we may find $\mathcal{N}\in\mathfrak{N}_{F,\sigma}(\X, V)$ with $\mathcal{N}(x)=\sum_{i=1}^M\mathcal{N}^{(i)}(x)v^{(i)},\, \mathcal{N}^{(i)}\in\mathfrak{N}_{F,\sigma}(\X),v^{(i)}\in V$ such that
    \[
    \norm{H\circ g(x)-\mathcal{N}(x)}_{V}<\varepsilon/3,\quad x\in \mathcal{K}.
    \]
    Since $H\circ g(\mathcal{K})$ is compact, we may find suitable $a_1,\dots,a_R\in H\circ g(\mathcal{K})$ such that 
    \[
    H\circ g(\mathcal{K})\subset\bigcup_{r=1}^RB(a_r,\varepsilon/3).
    \]
    Thus, for any $x\in\mathcal{K}$, there exists at least one $r_x\in\{1,\dots,R\}$ such that 
    \[
    \norm{H\circ g(x)-a_{r_x}}_{V}<\varepsilon/3
    \]
    and thus $\norm{\mathcal{N}(x)-a_{r_x}}_{V}<2\varepsilon/3$.
    We consider the metric projection $P_{a_1\dots,a_R}$ on $a_1,\dots,a_R$ which we know is $\mathcal{B}( V)/\mathcal{B}( V)$-measurable and with finite range (see Lemma \ref{lemma: projection in metric spaces}): so, the composition $P_{a_1\dots,a_R}\circ\mathcal{N}:\X\to V$ is $\mathcal{B}(\X)/\mathcal{B}( V)$-measurable. Fix $x\in\mathcal{K}$ and consider all $r\in\{1,\dots ,R\}$ such that $\norm{\mathcal{N}(x)-a_r}_{V}=d_{V}(\mathcal{N}(x),a_r)=d_{V}(\mathcal{N}(x),\{a_1,\dots,a_R\} )$. Since $\norm{\mathcal{N}(x)-a_{r_x}}_{V}<2\varepsilon/3$, it follows $d_{V}(\mathcal{N}(x),a_r)<2\varepsilon/3$ for all such $r$, and therefore
    \[
    \norm{\mathcal{N}(x)-P_{a_1,\dots,a_R}(\mathcal{N}(x))}_{V} < 2\varepsilon/3,\quad x\in\mathcal{K}.
    \]
    By the triangle inequality, we conclude
    \[
    \norm{H\circ g(x)-P_{a_1,\dots,a_R}(\mathcal{N}(x))}_{V} < \varepsilon,\quad x\in\mathcal{K}.
    \]
    By Lemma \ref{lemma: metrizability of compacts} we know that $H|_{g(\mathcal{K})}:(g(\mathcal{K}),\tau_{g(\mathcal{K})})\to H(g(\mathcal K))\subset V$ is an isometry, as well as its inverse. We set $\mathcal{M}:=H^{-1}\circ P_{a_1,\dots,a_R}\circ \mathcal{N}\in \mathfrak{BN}_{F,H,\sigma}(\X,\Y)$ and conclude that
    \[
    d_H\left(g(x),\mathcal M(x)\right)<\varepsilon,\quad x\in \mathcal{K}.
    \]

\end{proof}

\begin{remark}
    We observe that alternatively one may use in the proof of the last result neural network architectures of this form
    \[
    \sum_{k=1}^N\mathcal{N}_k (x)e_k,\quad N\in\N, (e_k)_{k=1}^\infty = \text{orthonormal basis of } V.
    \]
   This is clearly possible in virtue of Lemma \ref{lemma: E admits a pre-Schauder} (with now $E=V$).
\end{remark}

\begin{remark}\label{rmk: Acciaio's problem}
    Suppose now that we are given a compact metric space $(K,d)$ and a continuous map $g:(K,d)\to (Y,\rho)$, where $(Y,\rho)$ is a second metric space. Then $K$ is separable, and hence quasi-Polish. Besides, $g(K)$ is a compact, and thus $(g(K),\rho)$ is separable and once more quasi-Polish. Let $F$ and $H$ be the injection maps for $K$ and $g(K)$ respectively. Therefore, we may now apply the previous Proposition to the continuous map $g:(K,d)\to (g(K),\rho)$ and obtain a suitable architecture
    \[
    \mathcal M:=H^{-1}\circ P_{a_1,\dots,a_R}\circ \mathcal{N}:K\to g(K)\subset Y
    \]
    with $\mathcal{N}\in\mathfrak N_{F,\sigma}(K,V)$ and $a_1,\dots,a_R\in H\circ g(K)$ such that
    \[
    \rho(g(x), \mathcal{M}(x))<\varepsilon.
    \]
\end{remark}

Combining Proposition \ref{prop: target q.P} and Proposition \ref{prop: finite-dim approx for E t.v.s} immediately gives also in this case:
\begin{proposition}\label{prop: finite-dim approx for q.P.}
Assume in addition to the setting of Proposition \ref{prop: target q.P} that the activation function $\sigma$ is Lipschitz. Then given an arbitrary $g\in C(\X,\Y)$, a compact subset $\mathcal{K}\subset\X$ and an error $\varepsilon>0$, there exist 
\begin{itemize}
    \item $\mathcal{N}^{(1)},\dots,\mathcal{N}^{(M)}\in\mathfrak N_{F,\sigma}(\X)$ and $v^{(1)},\dots,v^{(M)}\in V$ for suitable $M\in\N$,
    \item $a_1,\dots,a_R\in H\circ g(\mathcal{K})\subset\mathcal{Q}\subset V$ for suitable $R\in\N$,
    \item and $\bar{N}\in\N$
\end{itemize}
such that for any $N\ge \Bar{N}$ it holds
\[
   d_H\left(g(x),\mathcal{M}
   (x)\right)<\varepsilon,\quad x\in \mathcal{K},
\]
where we have set $\mathcal{M}:=H^{-1}\circ P_{a_1,\dots,a_R}\circ \mathcal{N}    \in\mathfrak{BN}_{F,H,\sigma}(\X,\Y) $ and
\[
\mathcal{N}(x):=\sum_{i=1}^M \Lambda_N(\mathcal{N}^{(i)})(x) v^{(i)}.
\]
    
\end{proposition}

\section{On the necessity of the quasi-Polish condition}\label{sec: On the necessity of the quasi-Polish condition}

In this last section, our goal is to show a result which indicates that the category of quasi-Polish spaces is the correct category to work with if one aims at constructing approximating architectures on infinite-dimensional spaces $\X$ (topological dimension, algebraic dimension,...) which at the same time
\begin{enumerate}
    \item[i)] have sufficient expressive power to approximate arbitrary well continuous functions on $\X$,
    \item[ii)] are implementable in practice because specified by a finite number of parameters only,
    \item[iii)] and that are ``stable'' with respect to these parameters.
\end{enumerate}

These requirements are natural: clearly, the first one requires that the approximating architectures must satisfy universal approximation theorems of some sorts, while the second one merely demands that this family of functions must be represented and implementable into a machine with finite memory and computing power. The third one ensures some sort of continuity of the architectures with respect to theirs parameters, in the sense that tiny perturbations of the training parameters should be reflected in turn to tiny changes in the specification of the resulting architectures: refer to Definition \ref{def: continuity wrt param} and the discussion therein for extra details. 

After this premise, broadly speaking (see Proposition \ref{lemma: uat + finite} for a more precise statement), we will prove that if a topological space $(\X,\tau)$ grants the existence of such architectures, then it must be necessarily quasi-Polish.

We will focus here only on the scalar case, i.e. the target space on the architectures is $\R$, even though some ideas could also be extended to more general target spaces.

First of all, we define the ``infinite-dimensional parameters space''
\[
W:= V'\times \mathcal{L}(V)\times V
\]
endowed with the norm 
\[
\norm{(\phi,A,b)}_W := \norm{\phi}_{V'} + \norm{A}_{\mathcal{L}(V)} + \norm{b}_V,\quad (\phi,A,B)\in W.
\]
In this way, $(W,\norm{\cdot}_W)$ becomes a Banach space. We define the operator
\[
W\stackrel{R_1}{\longrightarrow} C(\X),\quad (\phi,A,b)\mapsto R_1(\phi,A,b) :=\mathcal{N}_{\phi,A,b}\circ F\equiv\langle \phi,\sigma(AF(\cdot)+b)
\rangle \in\mathfrak N_{F,\sigma}(\X),
\]
and, inductively, for $J\in\N$
\[
R_J:W^J\to C(\X),\quad R_J(w_1,\dots,w_J):= R_1(w_1)+\dots + R_1(w_J),\quad (w_1,\dots,w_J)\in W.
\]
So, evidently, $R_J((\phi_1,A_1,b_1),\dots,(\phi_J,A_J,b_j))=\sum_{j=1}^J\mathcal{N}_{\phi_j,A_j,b_j}\circ F$, namely the operators $R_J,J\in\N$ are the infinite-dimensional equivalent of the realization maps of \cite{PetersenRaslanVoigtlaender}. We observe
\begin{lemma}\label{lemma: the realization operator is continuous}
    Assume $\sigma:V\to V$ Lipschitz. For each $J\in\N$, the operator $R_J$ is continuous from $W^J$ into $C(\X)$, where $C(\X)$ is endowed with the topology of uniform convergence on compacts.
\end{lemma}
\begin{proof}
    Evidently, it is sufficient to prove the continuity of $R_1$. To this end, fix a compact set $\mathcal{K}\subset\X$ and an arbitrary point $(\phi,A,b)\in W$. It holds, for $x\in\mathcal K$ and $(\Bar{\phi},\bar{A},\bar{b})\in W$,
    \[
    \begin{split}
        &\abs{R_1(\phi,A,b)(x)-R_1(\Bar{\phi},\Bar{A},\Bar{b})(x)}=\abs{\langle\phi,\sigma(AF(x)+b)\rangle - \langle \Bar{\phi},\sigma(\Bar{A}F(x)+\Bar{b})\rangle}\\
        &\quad\le\abs{\langle \phi -\Bar{\phi},\sigma(AF(x)+b) \rangle} + \abs{\langle \Bar{\phi},\sigma(AF(x)+b) -\sigma(\Bar{A}F(x)+\bar{b})}\\
        &\quad\le \norm{\phi-\bar{\phi}}_V \max_{x\in\mathcal{K}}\norm{\sigma(AF(x)+b)}_V + Lip(\sigma)\norm{\Bar{\phi}}_{V'}\norm{AF(x)-\Bar{A}F(x)+b-\Bar{b}}_V\\
        &\quad\le\norm{\phi-\bar{\phi}}_V \max_{x\in\mathcal{K}}\norm{\sigma(AF(x)+b)}_V + Lip(\sigma)\norm{\Bar{\phi}}_{V'}\norm{A-\Bar{A} }_{\mathcal{L}(V)}\norm{F(x)}_V\\
        &\quad\quad\quad\quad +Lip(\sigma)\norm{\Bar{\phi}}_{V'}\norm{b-\bar{b}}_V. 
    \end{split}
    \]
    Because $\norm{F(x)}^2\le \sum_{i=1}^\infty \frac{1}{i^2} =
    \frac{\pi^2}{6}$ for any $x\in \X$, it is now clear that 
    \[
    \sup_{x\in\mathcal{K}}\abs{R_1(\phi,A,b)(x)-R_1(\Bar{\phi},\Bar{A},\Bar{b})(x)}\to 0
    \]
    as $(\Bar{\phi},\Bar{A},\Bar{b})\to (\phi,A,b)$ in $W$.
\end{proof}

Let us define the following extensions operators, which will allow us to canonically embed finite-dimensional parameters spaces into infinite-dimensional ones: to this end, let now $(e_j)_{j\in\N}$ be the canonical basis of $V$, and set $(N\in\N)$
\begin{equation}\label{eq: Ext1}
    \operatorname{Ext}_1:\R^N\to V',\quad h\mapsto \operatorname{Ext}_1(h);\; \langle\operatorname{Ext}_1(h),e_j\rangle:=
\begin{cases}
h_j,\quad 1\le j \le N\\
0,\quad\text{otherwise.}
\end{cases}
\end{equation}
and extended by linearity on the whole $V$.
\begin{equation}\label{eq: Ext2}
   \operatorname{Ext}_2:\R^{N\times N}\to\mathcal{L}(V),\quad \beta\mapsto\operatorname{Ext}_2(\beta):=B
\end{equation}
where $B$ is the unique element of $\mathcal{L}(V)$ such that $(Be_j,e_i)_V=\Bar{\beta}_{ij},\,j,i\in\N$, and where
\[
\Bar{\beta}_{ij}:=
\begin{cases}
\beta_{ij},\quad 1\le i,j\le N\\
0,\quad\text{otherwise}.
\end{cases}
\]
\begin{equation}\label{eq: Ext3}
    \operatorname{Ext}_3:\R^N\to V,\quad y\mapsto\operatorname{Ext}_3(y):=(y_1,\dots,y_N,0,0,\dots).
\end{equation}
These operators are well defined, linear and bounded: see Lemma \ref{lemma: Ext} in the Appendix. Besides, it is easy to check that the following identities hold,
\begin{equation}\label{eq: replication 1}
\operatorname{Ext}_1(\langle\phi,e_1\rangle,\dots,\langle\phi,e_N\rangle) = \phi\circ\Pi_N,\quad\forall\phi\in V',
\end{equation}
\begin{equation}\label{eq: replication 2}
\operatorname{Ext}_2(\beta) = \Pi_N\circ A\circ \Pi_N,\quad \forall A\in\mathcal{L}(V),\quad\text{where } \beta_{ij}:= (Ae_j,e_i)_V,\,1\le i,j\le N,
\end{equation}
\begin{equation}\label{eq: replication 3}
\operatorname{Ext}_3((b,e_1)_V,\dots,(b,e_N)_V) =\Pi_Nb,\quad \forall b\in V.
\end{equation}

Consequently we define the following extension operator from a ``finite-dimensional parameters space'' into the ``infinite-dimensional parameters space'' $W$
\[
\operatorname{Ext}: \R^N\times\R^{N\times N}\times \R^N\to W,\quad (h,\beta,y)\mapsto (\operatorname{Ext}_1(h),\operatorname{Ext}_2(\beta),\operatorname{Ext}_3(y)),
\]
which is clearly linear and bounded, because
\[
\norm{\operatorname{Ext}(h,\beta,y)}_W\le \norm{h}_{\R^N}+\norm{\beta}_{\R^{N\times N}} + \norm{y}_{\R^N} 
\]
(refer to the proof of Lemma \ref{lemma: Ext}).

Similarly, for $J\in\N$ we can also define the $J$-th tensor power of this operator in the natural way
\[
\begin{split}
    &\operatorname{Ext}^{\otimes J}: \left(\R^N\times\R^{N\times N}\times \R^N\right)^J\to W^J\\
    &((h_1,\beta_1,y_1),\dots,(h_J,\beta_J,y_J))\mapsto (\operatorname{Ext}(h_1,\beta_1,y_1),\dots ,\operatorname{Ext}(h_J,\beta_J,y_J))
\end{split}
\]

which is again linear and bounded. In view of all of this and Lemma \ref{lemma: the realization operator is continuous}, we then have
\begin{lemma}\label{lemma: the realization operator comcposed with Ext is continuous}
    Assume $\sigma:V\to V$ Lipschitz. For each $N,J\in\N$, the non-linear ``realization'' operator 
    \[
    R_J\circ \operatorname{Ext}^{\otimes J}: \left(\R^N\times\R^{N\times N}\times \R^N\right)^J\to C(\X) 
    \]
    is continuous, where $C(\X)$ is endowed with the topology of uniform convergence on compacts.
\end{lemma}
More explicitly, the action of this operator is
\[
R_J\circ \operatorname{Ext}^{\otimes J}((h_1,\beta_1,y_1),\dots,(h_J,\beta_J,y_J)) = \sum_{j=1}^J \langle \operatorname{Ext}_1(h_j),\sigma(\operatorname{Ext}_2(\beta_j)F(\cdot) +\operatorname{Ext}_3(y_j) )\rangle.
\]

Therefore, the continuous operator $R_J\circ \operatorname{Ext}^{\otimes J}$ naturally induces the following map: set for convenience $r:=(N^2+2N)J$, and define
\begin{equation}\label{eq: phi r from my NN}
\phi^{(r)}:\X\times\R^r\to\R,\quad (x,\theta)\mapsto (R_J\circ\operatorname{Ext}^{\otimes J})(\theta)(x).
\end{equation}
From the last Lemma, we see that the mapping
    \[
    \R^r\ni\theta\mapsto\phi^{(r)}(\cdot,\theta)\in C(\X)
    \]
    is continuous, and, from Theorems \ref{thm: first} and \ref{thm: second}, we deduce that, for any $\varepsilon>0,\mathcal{K}\subset\X$ compact and $g:\X\to\R$ continuous there exist $r\in\N$ and $\theta\in\R^r$ such that
    \[
    \sup_{x\in\mathcal{K}}\abs{g(x) - \phi^{(r)}(x,\theta)}<\varepsilon.
    \]

All of that motivates the following. Let us consider now an arbitrary topological space $(\X,\tau)$. For any $r\in\N$ consider the following ``architectures''
\[
\phi^{(r)}:\X\times \Theta^{(r)}\to \R
\]
where $\Theta^{(r)}\subset\R^r$ is an arbitrary non-empty subset of the Euclidean space. Assume that $\phi^{(r)}(\cdot,\theta)\in C(\X)$ for any $\theta\in\Theta^{(r)}$, and observe that we do not require the same functional form as above, i.e. we are not necessarily considering a parametric family. Set
\begin{equation}\label{eq: Phi r}
\Phi^{(r)}:=\{\phi^{(r)}(\cdot,\theta);\,\theta\in\Theta^{(r)}\}\subset C(\X)
\end{equation}
and
\begin{equation}\label{eq: Phi}
\Phi:=\bigcup_{r\in\N}\Phi^{(r)}
\end{equation}
We give the following definition:
\begin{definition}\label{def: (UAP)}
    We say that the Universal Approximation Property (UAP) holds for the family $\Phi$ if for any $\varepsilon>0$, for any $\mathcal{K}\subset\X$ compact and any $g\in C(\X)$ there exists $u\in\Phi$ such that \[
\sup_{x\in\mathcal{K}}\abs{g(x)-u(x)}<\varepsilon,
\]
namely $\Phi$ is dense in $C(\X)$, whereas also in this case $C(\X)$ is endowed with the locally convex topology of the uniform convergence on compacts, namely the one generated by the seminorms $$\{p_\mathcal{K};\,\mathcal{K}\subset\X\text{ compact}\},$$ and where obviously $p_\mathcal{K}(g):=\sup_{x\in\mathcal{K}}\abs{g(x)}$. 
\end{definition}

Besides, 
\begin{definition}\label{def: continuity wrt param}
  We say that the family $\Phi$ is continuous with respect to its parameters if for any $r\in\N$, the mappings
    \[
    \Theta^{(r)}\ni\theta\mapsto\phi^{(r)}(\cdot,\theta)\in C(\X)
    \]
    are continuous  
\end{definition}

As we have just seen above, our infinite-dimensional architectures satisfy this stability property. Besides, from \cite[Proposition 4.1]{PetersenRaslanVoigtlaender}, we also see that all classical feedforward neural networks enjoy this property, as soon as the activation function is assumed to be continuous. This property seems to be natural and desirable in practice, because it ensures that small perturbations of the training parameters will not produce dramatic changes in the realization of the final architectures.     

The property is also ensured for instance if $\X$ is a metric space and $\phi^{(r)}$ is assumed jointly continuous. Indeed, let us fix an arbitrary compact set $\mathcal{K}\subset\X$ and a point $\theta_0\in \Theta^{(r)}$. For convenience, let us also assume that $\Theta^{(r)}$ is open. Let $\varepsilon>0$. Then $\mathcal{K}\times B$ is compact, where $B$ is a suitable closed ball in $\Theta^{(r)}$ around $\theta_0$. By Heine-Cantor, $\phi^{(r)}$ is uniformly continuous on $\mathcal{K}\times B$, and this leads to 
    \[
    \abs{\phi^{(r)}(x,\theta) -\phi^{(r)}(x,\theta_0) } < \varepsilon,\quad x\in\mathcal{K},\;\abs{\theta-\theta_0}<\delta
    \]
    for a suitable $\delta>0$, namely $p_{\mathcal K}(\phi^{(r)}(\cdot,\theta) - \phi^{(r)}(\cdot,\theta_0))<\varepsilon$ if $\abs{\theta-\theta_0}<\delta$: we have showed that $\phi^{(r)}(\cdot,\theta)$ is continuous at an arbitrary point $\theta_0$, and hence the claim.

Therefore, we can re-formulate our results from the previous sections as:
\begin{proposition}\label{prop: re-formulation of our results}
    Let $(\X,\tau)$ be a quasi-Polish space with injection map $F$. Assume that the activation function $\sigma:V\to V$ is Lipschitz continuous, satisfying the separating property \eqref{eq: abstract condition on sigma} and with bounded range. Consider  
    \[
\Phi^{(r)}=\{\phi^{(r)}(\cdot,\theta);\,\theta\in\R^r\},\quad \Phi=\bigcup_{r\in\N}\Phi^{(r)}
\]
where the maps $\phi^{(r)}:\X\times\R^r\to \R$ are defined in \eqref{eq: phi r from my NN}
Then the family $\Phi$ enjoys (UAP) and it is continuous with respect to its parameters.
\end{proposition}

We are now going to show a result which is basically the converse of this last one, and which somehow suggests that the category of quasi-Polish spaces is the correct category to work with in order to obtain universal approximation theorems in infinite-dimensional spaces (topological, algebraic,...) and approximating architectures that are  nonetheless implementable in practice because specified by a finite number of parameters.

We have:
\begin{proposition}\label{lemma: uat + finite}
    Let $(\X,\tau)$ be a Tychonoff topological space, namely $\X$ is Hausdorff and completely regular. Assume that there exists a family of functions $\Phi$ defined as in \eqref{eq: Phi} that enjoys the (UAP) and that is continuous with respect to its parameters. Then $(\X,\tau)$ is quasi-Polish.
\end{proposition}
\begin{proof}
   Consider $D^{(r)}\subset\Theta^{(r)}$ dense and countable, and define
   \[
\mathcal{D}^{(r)}=\{\phi^{(r)}(\cdot,\theta);\,\theta\in D^{(r)}\}\subset \Phi^{(r)}
\]
and
\[
\mathcal{D}=\bigcup_{r\in\N}\mathcal{D}^{(r)}
\subset\Phi.
\]
First of all, we want to ``transport'' (UAP) on $\mathcal{D}$: to this end, we choose an arbitrary $\varepsilon>0$, a compact subset $\mathcal{K}\subset\X$ and $g\in C(\X)$. Then there exists $u\in\Phi$ such that
\[
p_\mathcal{K}(g-u)<\varepsilon/2.
\]
So, $u(\cdot)=\phi^{(r_0)}(\cdot,\theta_0)$ for some $r_0\in\N$ and $\theta_0\in\Theta^{(r_0)}$. Moreover, in view of the continuity of $\Phi$ with respect to its parameters, there exists $\delta>0$ such that 
\[
 p_\mathcal{K}(\phi^{(r_0)}(\cdot,\theta) - \phi^{(r_0)}(\cdot,\theta_0)  ) < \varepsilon/2
\]
whenever $\abs{\theta-\theta_0}<\delta,\,\,\theta\in\Theta^{(r_0)}$. By density, we can and will choose $\abs{\theta-\theta_0}<\delta,\,\,\theta\in D^{(r_0)} $. Hence, we get $$p_\mathcal{K}(\phi^{(r_0)}(\cdot,\theta) - \phi^{(r_0)}(\cdot,\theta_0)  ) < \varepsilon/2,$$ but where now $\phi^{(r_0)}(\cdot,\theta)\in\mathcal{D}^{(r_0)} \subset \mathcal{D}$ rather than in the larger space $\Phi^{(r_0)}$. Hence,
\[
\begin{split}
    p_{\mathcal{K}}(g-\phi^{(r_0)}(\cdot,\theta)) &\le p_{\mathcal{K}}(g-u) + p_{\mathcal{K}}(u-\phi^{(r_0)}(\cdot,\theta))\\
    &< \varepsilon/2 + p_{\mathcal{K}}(\phi^{(r_0)}(\cdot,\theta_0)  -\phi^{(r_0)}(\cdot,\theta))\\
    &<\varepsilon,
\end{split}
\]
and thus we see that the (UAP) holds also for the smaller class $\mathcal{D}\subset C(\X)$.

Our goal is now to prove that $\mathcal{D}$ separates the points of $\X$. To this end, assume that there exist $x_0,x_1\in\X,x_0\neq x_1$ such that $u(x_0)=u(x_1)$ for all $u\in\mathcal{D}$. We observe that $\{x_1\}$ is closed, because $\X$ is Hausdorff, and, since it is completely regular, we may find $g\in C(\X)$ such that $g(x_1)=1$ and $g(x_0)=0$. We fix $\varepsilon=1/4$ and the compact set $\mathcal{K}=\{x_0,x_1\}$. By the (UAP) for $\mathcal{D}$, there must exist $\Tilde{u}\in \mathcal{D}$ such that 
\[
\abs{1-\Tilde{u}(x_1)}<1/4,\quad \abs{-\Tilde{u}(x_0)}<1/4.
\]
But $\Tilde{u}(x_1)=\Tilde{u}(x_0)$, and so we obtain $1\le \abs{1-\Tilde{u}(x_1)}+\abs{\Tilde{u}(x_1)}<1/2 $: contradiction! We conclude that $\mathcal{D}$ must separate the points of $\X$. Finally, by construction, $\mathcal{D}$ is countable, and thus it may serve as a separating sequence for $(\X,\tau)$. 
\end{proof}

\section{Appendix}\label{sec: Appendix}

\begin{lemma}\label{lemma: Ext}
    Consider the operators $\operatorname{Ext}_k,k=1,2,3$ defined in \eqref{eq: Ext1}, \eqref{eq: Ext2} and \eqref{eq: Ext3}. These operators are well-defined, linear and bounded.
\end{lemma}
\begin{proof}
   It is trivial to check that $\operatorname{Ext}_1(h)\in V'$ for any $h\in\R^N$. Besides, $\operatorname{Ext}_1$ is linear and 
   \[
   \norm{\operatorname{Ext}_1(h)}_{V'}=\sup_{\norm{v}_V=1 }\abs{\langle \operatorname{Ext}_1(h),v \rangle}\le \norm{h}_{\R^N},
   \]
showing its boundedness. 

Let us deal with $\operatorname{Ext}_2$ now. Given $\beta\in\R^{N\times N}$, define $\Bar{\beta}$ accordingly. Then define $Be_j,j\in\N$ as the unique element of $V$ such that $(Be_j,e_i)_V=\Bar{\beta}_{ij},\,j,i\in\N$. Besides, it holds
\[
\sum_{j=1}^\infty\norm{Be_j}_V = \sum_{j=1}^N\left\{\sum_{i=1}^N\beta^2_{ij} \right\}^{1/2}<\infty,
\]
which guarantees that there exists a unique extension of $B$ to a bounded linear operator from $V$ to itself: see e.g. \cite[page 30]{conway2010}. We set $\operatorname{Ext}_2(\beta):=B$. It is easy to check that $\operatorname{Ext}_2$ is linear. Moreover, for each $v\in V,v=\sum_{j=1}^\infty v_je_j$, it holds $Bv = \sum_{j=1}^\infty v_jBe_j$, and whence
\[
(Bv,e_i)_V = \sum_{j=1}^\infty v_j(Be_j,e_i)_V=\sum_{j=1}^\infty v_j\Bar{\beta}_{ij} =\sum_{j=1}^Nv_j\beta_{ij}
\]
for $1\le i\le N$ and 0 otherwise. We conclude that
\[
\norm{Bv}_V\le \left[
\norm{v}_V^2\sum_{i,j=1}^N\beta_{ij}^2
\right]^{1/2}=
\norm{v}_V\norm{\beta}_{\R^{N\times N}}
,\quad v\in V,
\]
i.e. $\operatorname{Ext}_2$ is bounded.

Finally, the fact the $\operatorname{Ext}_3$ is linear and bounded is trivial.
   
\end{proof}

\subsection{On the reasons for the choice of a non-standard metric projection}\label{subsec: metric projection} 

In this last part, we are going to explain the reasons why we had to resort to the special metric projection defined in Lemma \ref{lemma: projection in metric spaces}, which is not continuous but only measurable, rather than to the standard metric projection (also known as best approximation) and continuous selections: see e.g. \cite{Deutsch}. For unexplained terminology in the sequel, we refer to that paper.

For an arbitrary non-empty set $M\subset V$ and $x\in V$, we define the set of all best approximations to $x$ from $M$ as
\[
\mathcal{P}_M(x)=\{
y\in M;\, \norm{x-y}=\inf_{z\in M}\norm{x-z}
\}.
\]
Consider now $M=\{a_1,\dots, a_R\}\subset V$ with $R> 2$. Then it is easy to see that $M$ is proximinal. Besides, we claim that $M$ is almost Chebyshev. Indeed, we have
\[
I:=\{
x\in V;\, \mathcal P_M(x) \text{ is not a singleton}  
\}=
\bigcup_{A\subset M;\, \sharp(A)\ge 2}
\{
x\in V;\,\mathcal P_M(x)=A
\}
\]
Consider now $A=\{a_i,a_j\}$ and $x\in V:\, \mathcal P_M(x)=A$. Then, we must have $\norm{x-a_i}=\norm{x-a_j}$, from which we deduce
\[
( x,a_i-a_j)_V = \frac{1}{2}(\norm{a_i}^2-\norm{a_j}^2).
\]
It follows, 
\[
\left\{
x\in V;\, \mathcal P_M(x)=\{a_i,a_j\}
\right\} \subset \left\{ x\in V;\,
( x,a_i-a_j)_V = \frac{1}{2}(\norm{a_i}^2-\norm{a_j}^2)
\right\},
\]
and the right hand side is closed and with empty interior (since $a_i\neq a_j$), namely it is nowhere dense.

On the other hand, if $\sharp(A)>2$ and $\mathcal P_M(x)=A$, then 
\[
x\in \left\{ x\in V;\,
( x,a_i-a_j)_V = \frac{1}{2}(\norm{a_i}^2-\norm{a_j}^2)
\right\}
\]
for any possible choice of $a_i\neq a_j\in A$. In light of this, we can write now
\[
I \subset \cup_{i\neq j} \left\{ x\in V;\,
( x,a_i-a_j)_V = \frac{1}{2}(\norm{a_i}^2-\norm{a_j}^2)
\right\},
\]
namely $I$ is a subset of a meager set, and so it is itself meager. So, $M$ is almost Chebyshev. 

Let us prove now that $\mathcal P_M$ is not 2-lower-semicontinuous: upon re-labelling, we can assume from the beginning
\[
    \alpha:=\norm{a_1-a_2}\le \norm{a_j-a_\ell},\quad j\neq\ell. 
\]
Pick $0<\varepsilon<1$ such that 
\[
    B_\varepsilon(a_1) \cap B_\varepsilon(a_2) = \emptyset.
\]
Let $x_0\in V$ and $U(x_0)$ be an arbitrary neighborhood of $x_0$. Choose $1/2<t<1$ and $0<s<1/2$ such that
    \[
    x_t:=ta_1+(1-t)a_2\in U(x_0), \quad x_s:=sa_1+(1-s)a_s\in U(x_0).
    \]
    Thus, $\norm{a_1-x_t}=(1-t)\alpha$, $\norm{a_2-x_t}=t\alpha$, and for $j\neq 1,2$ it holds
    \[
    \norm{a_j-x_t}\ge\abs{\norm{a_j-a_1} - \norm{a_1-x_t} } = \abs{\norm{a_j-a_1} -(1-t)\alpha } = \norm{a_j-a_1}-(1-t)\alpha
    \]
    because $\norm{a_j-a_1}\ge\alpha$.  
    In light of this, $\norm{a_j-x_t}\ge t\alpha$. Since $1/2<t$, we conclude $\mathcal P_M(x_t)=\{a_1\}$. By symmetry, $\mathcal P_M(x_s)=\{a_2\}$. Thus,
    \[
    B_\varepsilon(\mathcal P_M(x_t))\cap B_\varepsilon(\mathcal P_M(x_s))=B_\varepsilon(a_1)\cap B_\varepsilon(a_2)
    \]
    which is the empty set. Therefore, $\mathcal P_M$ is not 2-lower-semicontinous at $x_0$.

So, to sum up, $M=\{a_1,\dots,a_R\}$ is proximinal, almost Chebyshev, and the projection $\mathcal{P}_M$ is not 2-lower-semicontinuous. But then, in force of the characterization provided by \cite[Corollary 3.2]{Deutsch}, we conclude that $\mathcal P_M$ can not admit a continuous selection, namely a continuous map $p:V\to M$ such that for all $x\in V$ it holds $p(x)\in\mathcal P_M(x)$.

\vskip 1cm
\noindent {\bf Conflict of Interest:} The author declares that he has no conflict of interest.

\bibliography{literature}

\begin{thebibliography}{10}

\bibitem{Hypertransformer}
B.~Acciaio, A.~Kratsios, and G.~Pammer.
\newblock Designing universal causal deep learning models: The geometric
  (hyper)transformer.
\newblock {\em Mathematical Finance}, 34:671--735, 2024.

\bibitem{256500}
A.~Barron.
\newblock Universal approximation bounds for superpositions of a sigmoidal
  function.
\newblock {\em IEEE Transactions on Information Theory}, 39(3):930--945, 1993.

\bibitem{beck2021overview}
C.~Beck, M.~Hutzenthaler, A.~Jentzen, and B.~Kuckuck.
\newblock An overview on deep learning-based approximation methods for partial
  differential equations, 2021.

\bibitem{BDG0}
F.~Benth, N.~Detering, and L.~Galimberti.
\newblock Neural networks in {F}r\'echet spaces.
\newblock {\em Annals of Mathematics and Artificial Intelligence}, 91:75--103,
  2023.

\bibitem{BDG2}
F.~E. Benth, N.~Detering, and L.~Galimberti.
\newblock Pricing options on flow forwards by neural networks in {H}ilbert
  space.
\newblock {\em Finance and Stochastic}, 28:81--121, 2024.

\bibitem{MathFoundationsDeep}
J.~Berner, P.~Grohs, G.~Kutyniok, and P.~Petersen.
\newblock {\em The Modern Mathematics of Deep Learning}.
\newblock Cambridge University Press, 2022.

\bibitem{Billingsley}
P.~Billingsley.
\newblock {\em Probability and Measure}.
\newblock Wiley, 1995.

\bibitem{BrzezniakMotyl}
Z.~Brzezniak and E.~Motyl.
\newblock Existence of a martingale solution of the stochastic navier-stokes
  equation in unbounded 2d and 3d domains.
\newblock {\em J. Differ. Equ.,}, 254(4):1627--1685, 2013.

\bibitem{BrzezniakOndrejat}
Z.~Brzezniak and M.~Ondrejat.
\newblock Weak solutions to stochastic wave equations with values in riemannian
  manifolds.
\newblock {\em Commun. Partial Differ. Equ.,}, 36(9):1624--1653, 2011.

\bibitem{286886}
T.~Chen and H.~Chen.
\newblock Approximations of continuous functionals by neural networks with
  application to dynamic systems.
\newblock {\em IEEE Transactions on Neural Networks}, 4(6):910--918, 1993.

\bibitem{392253}
T.~Chen and H.~Chen.
\newblock Universal approximation to nonlinear operators by neural networks
  with arbitrary activation functions and its application to dynamical systems.
\newblock {\em IEEE Transactions on Neural Networks}, 6(4):911--917, 1995.

\bibitem{ChoSaul}
Y.~Cho and L.~K. Saul.
\newblock Kernel methods for deep learning.
\newblock {\em Advances in neural information processing systems,}, pages
  342--350, 2009.

\bibitem{conway2010}
J.~B. Conway.
\newblock {\em A {C}ourse in {F}unctional {A}nalysis}.
\newblock Graduate {T}exts in {M}athematics; 96. Springer Science+Business
  Media, New York, 2nd edition, 2010.

\bibitem{Cuchiero2020DeepNN}
C.~Cuchiero, M.~Larsson, and J.~Teichmann.
\newblock Deep neural networks, generic universal interpolation, and controlled
  odes.
\newblock {\em SIAM J. Math. Data Sci.}, 2:901--919, 2020.

\bibitem{CuchieroSchmockerTeichmann}
C.~Cuchiero, P.~Schmocker, and J.~Teichmann.
\newblock Global universal approximation of functional input maps on weighted
  spaces, 2023.

\bibitem{Cybenko1989}
G.~Cybenko.
\newblock Approximation by superpositions of a sigmoidal function.
\newblock {\em Mathematics of Control, Signals and Systems}, 2(4):303--314,
  1989.

\bibitem{daprato_zabczyk_1992}
G.~Da~Prato and J.~Zabczyk.
\newblock {\em Stochastic Equations in Infinite Dimensions}.
\newblock Encyclopedia of Mathematics and its Applications. Cambridge
  University Press, 1992.

\bibitem{Deutsch}
F.~Deutsch and P.~Kenderov.
\newblock Continuous selections and approximate selection for set-valued
  mappings and applications to metric projections.
\newblock {\em SIAM J. Math. Analysis,}, 14(1), 1983.

\bibitem{10.1007/BF02392270}
P.~Enflo.
\newblock {A counterexample to the approximation problem in Banach spaces}.
\newblock {\em Acta Mathematica}, 130:309 -- 317, 1973.

\bibitem{FUNAHASHI1989183}
K.-I. Funahashi.
\newblock On the approximate realization of continuous mappings by neural
  networks.
\newblock {\em Neural Networks}, 2(3):183--192, 1989.

\bibitem{GalimbertiHoldenKarlsenPang}
L.~Galimberti, H.~Holden, K.~Karlsen, and P.~Pang.
\newblock Global existence of dissipative solutions to the camassa–holm
  equation with transport noise.
\newblock {\em Journal of Differential Equations}, 387:1--103, 2024.

\bibitem{GKL}
L.~Galimberti, A.~Kratsios, and G.~Livieri.
\newblock Designing universal causal deep learning models: The case of
  infinite-dimensional dynamical systems from stochastic analysis, 2022.

\bibitem{PreSchauder}
F.~Garcia-Pacheco and F.~Perez-Fernandez.
\newblock Pre-schauder bases in topological vector spaces.
\newblock {\em Symmetry}, 11(8), 2019.

\bibitem{Guss}
W.~H. Guss and R.~Salakhutdinov.
\newblock On universal approximation by neural networks with uniform guarantees
  on approximation of infinite dimensional maps, 2019.

\bibitem{Han8505}
J.~Han, A.~Jentzen, and W.~E.
\newblock Solving high-dimensional partial differential equations using deep
  learning.
\newblock {\em Proceedings of the National Academy of Sciences},
  115(34):8505--8510, 2018.

\bibitem{articleWidth2}
B.~Hanin and M.~Sellke.
\newblock Approximating continuous functions by relu nets of minimal width.
\newblock {\em ArXiv}, abs/1710.11278, 2017.

\bibitem{HJ}
T.~Hazan and T.~Jaakola.
\newblock Steps toward deep kernel methods from infinite neural networks, 2015.

\bibitem{Hornik1991}
K.~Hornik.
\newblock Approximation capabilities of multilayer feedforward networks.
\newblock {\em Neural Networks}, 4(2):251--257, 1991.

\bibitem{HORNIK1989359}
K.~Hornik, M.~Stinchcombe, and H.~White.
\newblock Multilayer feedforward networks are universal approximators.
\newblock {\em Neural Networks}, 2(5):359--366, 1989.

\bibitem{dneuralnetPDE}
M.~Hutzenthaler, A.~Jentzen, T.~Kruse, and T.~A. Nguyen.
\newblock A proof that rectified deep neural networks overcome the curse of
  dimensionality in the numerical approximation of semilinear heat equations.
\newblock {\em SN Partial Differential Equations and Applications}, 1(2):10,
  2020.

\bibitem{jakubowski}
A.~Jakubowski.
\newblock Short communication:the almost sure skorokhod representation for
  subsequences in nonmetric spaces.
\newblock {\em Theory of Prob. and Its Appl.}, 42(1), 1998.

\bibitem{kechris}
A.~Kechris.
\newblock {\em Classical Descriptive Set Theory}.
\newblock Springer, 1995.

\bibitem{pmlr-v125-kidger20a}
P.~Kidger and T.~Lyons.
\newblock {Universal Approximation with Deep Narrow Networks}.
\newblock In J.~Abernethy and S.~Agarwal, editors, {\em Proceedings of Thirty
  Third Conference on Learning Theory}, volume 125 of {\em Proceedings of
  Machine Learning Research}, pages 2306--2327. PMLR, 09--12 Jul 2020.

\bibitem{https://doi.org/10.48550/arxiv.2108.08481}
N.~Kovachki, Z.~Li, B.~Liu, K.~Azizzadenesheli, K.~Bhattacharya, A.~Stuart, and
  A.~Anandkumar.
\newblock Neural operator: Learning maps between function spaces, 2021.

\bibitem{EchoState}
G.~L. and O.~J.P.
\newblock Echo state networks are universal.
\newblock {\em Neural Networks}, 108:495--508, 2018.

\bibitem{ReservoirComputing}
G.~L., G.~L., and O.~J.P.
\newblock Infinite-dimensional reservoir computing, 2023.

\bibitem{10.1093/imatrm/tnac001}
S.~Lanthaler, S.~Mishra, and G.~E. Karniadakis.
\newblock {Error estimates for DeepONets: a deep learning framework in infinite
  dimensions}.
\newblock {\em Transactions of Mathematics and Its Applications}, 6(1), 03
  2022.
\newblock tnac001.

\bibitem{LESHNO1993861}
M.~Leshno, V.~Y. Lin, A.~Pinkus, and S.~Schocken.
\newblock Multilayer feedforward networks with a nonpolynomial activation
  function can approximate any function.
\newblock {\em Neural Networks}, 6(6):861--867, 1993.

\bibitem{https://doi.org/10.48550/arxiv.2003.03485}
Z.~Li, N.~Kovachki, K.~Azizzadenesheli, B.~Liu, K.~Bhattacharya, A.~Stuart, and
  A.~Anandkumar.
\newblock Neural operator: Graph kernel network for partial differential
  equations, 2020.

\bibitem{FNO}
Z.~Li, K.~N., K.~Azizzadenesheli, B.~Liu, K.~Bhattacharya, A.~Stuart, and
  A.~Anima.
\newblock Fourier neural operator for parametric partial differential
  equations.
\newblock {\em ICLR}, 2019.

\bibitem{LiDeepONets}
L.~Lu, P.~Jin, G.~Pang, Z.~Zhang, and G.~E. Karniadakis.
\newblock Learning nonlinear operators via deeponet based on the universal
  approximation theorem of operators.
\newblock {\em Nature Machine Intelligence}, 3(3):218--229, 2021.

\bibitem{articleWidth1}
Z.~Lu, H.~Pu, F.~Wang, Z.~Hu, and L.~Wang.
\newblock The expressive power of neural networks: A view from the width.
\newblock {\em Neural Information Processing Systems}, 09 2017.

\bibitem{PhysicsInformed}
R.~M., P.~P., and K.~G.
\newblock Physics-informed neural networks: A deep learning framework for
  solving forward and inverse problems involving nonlinear partial differential
  equations.
\newblock {\em Journal of Computational Physics}, 378:686--707, 2019.

\bibitem{McCullochPitts}
W.~McCulloch and W.~Pitts.
\newblock A logical calculus of the ideas immanent in nervous activity.
\newblock {\em The bulletin of mathematical biophysics}, 5:115--133, 1943.

\bibitem{6795742}
H.~Mhaskar and H.~Nahmwoo.
\newblock Neural networks for functional approximation and system
  identification.
\newblock {\em Neural Computation}, 9(1):143--159, 1997.

\bibitem{https://doi.org/10.48550/arxiv.2201.11440}
D.~Müller, I.~Soto-Rey, and F.~Kramer.
\newblock An analysis on ensemble learning optimized medical image
  classification with deep convolutional neural networks, 2022.

\bibitem{Neal}
R.~M. Neal.
\newblock {\em Bayesian {L}earning for {N}eural {N}etworks}.
\newblock Lecture {N}otes in {S}tatistics: 118. Springer Science+Business
  Media, New York, 1996.

\bibitem{PetersenRaslanVoigtlaender}
P.~Petersen, M.~Raslan, and F.~Voigtlaender.
\newblock Topological properties of the set of functions generated by neural
  networks of fixed size.
\newblock {\em Foundations of Computational Mathematics}, 21:375--444, 2021.

\bibitem{pinkus_1999}
A.~Pinkus.
\newblock Approximation theory of the mlp model in neural networks.
\newblock {\em Acta Numerica}, 8:143–195, 1999.

\bibitem{RamSilver}
J.~O. Ramsey and B.~W. Silverman.
\newblock {\em Functional {D}ata {A}nalysis}.
\newblock Springer Science+Business Media, New York, 2nd edition, 2005.

\bibitem{Mangatiana_CompactnessPrinciplesTVS}
M.~Robdera.
\newblock Compactness principles for topological vector spaces.
\newblock {\em Topological Algebra and its Applications}, 10:246--254, 2022.

\bibitem{Robinson}
J.~Robinson.
\newblock {\em Dimensions, embeddings, and attractors}.
\newblock Cambridge Tracts in Mathematics. Cambridge University Press, 2011.

\bibitem{76498}
I.~Sandberg.
\newblock Approximation theorems for discrete-time systems.
\newblock {\em IEEE Transactions on Circuits and Systems}, 38(5):564--566,
  1991.

\bibitem{schaefer1999topological}
H.~Schaefer.
\newblock {\em Topological Vector Spaces (Second Edition)}.
\newblock Springer, 1999.

\bibitem{shuchat}
A.~Shuchat.
\newblock Approximation of vector-valued continuous functions.
\newblock {\em Proc. of the American Mathematical Society}, 31(1), 1972.

\bibitem{shuchat2}
A.~Shuchat.
\newblock Integral representation theorems in topological vector spaces.
\newblock {\em Trans. of the American Mathematical Society}, 172, 1972.

\bibitem{SmithTrivisa}
S.~Smith and K.~Trivisa.
\newblock The stochastic navier-stokes equations for heat-conducting,
  compressible fluids: global existence of weak solutions.
\newblock {\em J. Evol. Equ.,}, 18(2):411--465, 2018.

\bibitem{STINCHCOMBE1999467}
M.~Stinchcombe.
\newblock Neural network approximation of continuous functionals and continuous
  functions on compactifications.
\newblock {\em Neural Networks}, 12(3):467--477, 1999.

\bibitem{brainImage}
T.~S. Tian.
\newblock Functional data analysis in brain imaging studies.
\newblock {\em Frontiers in psychology}, 1:35--35, 10 2010.

\bibitem{Williams}
C.~K.~I. Williams.
\newblock Computing with infinite networks.
\newblock {\em Advances in neural information processing systems,}, pages
  295--301, 1997.

\end{thebibliography}
\bibliographystyle{abbrv}

\end{document}